\documentclass[11pt,twoside,reqno,nosumlimits]{amsart}
\usepackage{subeqnarray,url}

\usepackage{etoolbox}
\usepackage{amssymb}
\usepackage{newtxtext,newtxtt}   
\usepackage{microtype}

\usepackage[toc,page]{appendix}
\usepackage{fancyvrb}
\usepackage{listings}
\usepackage{algorithm}
\usepackage{algpseudocode}
\usepackage{bm}

\usepackage{mathtools}

\patchcmd{\section}{\scshape}{\bfseries}{}{}
\makeatletter
\renewcommand{\@secnumfont}{\bfseries}
\makeatother
\usepackage[toc,page]{appendix}
\usepackage{todonotes}

\usepackage[dvipsnames]{xcolor}
\usepackage{siunitx}
\patchcmd{\section}{\normalfont}{\normalfont\color{MidnightBlue}}{}{}
\patchcmd{\subsection}{\normalfont}{\normalfont\color{MidnightBlue}}{}{}
\makeatletter
\def\subsubsection{\@startsection{subsubsection}{3}%
\z@{.5\linespacing\@plus.7\linespacing}{-.5em}%
{\normalfont\bfseries}}
\makeatother

\usepackage{circuitikz}
\usepackage{fancyhdr}
\usepackage{amsmath,amsfonts,amsbsy,amsgen,amscd,mathrsfs,amsthm,mathtools,tensor}

\usepackage[a4paper,margin=2.0cm]{geometry}
\usepackage{tikz}
\usepackage{overpic}
\usepackage{amsopn}

\usepackage{array} 
\usepackage{booktabs}
\usepackage{makecell}
\usepackage{caption}
\usepackage{graphicx}
\usepackage{epstopdf}
\usepackage{pdfsync}
\usepackage[colorlinks]{hyperref}
\usepackage{comment}
\usepackage{mathabx}
\usepackage{upgreek}
\usepackage{subfig}
\usepackage{parskip}
\usepackage{cleveref}
\usepackage{yfonts}
\usepackage{epsfig,multirow}

\usepackage{accents}

\newtheorem{theorem}{Theorem}[section]

\newtheorem{assumption}[theorem]{Assumption}

\theoremstyle{definition}
\newtheorem{definition}[theorem]{Definition}
\theoremstyle{remark}
\newtheorem{remark}[theorem]{Remark}
\newtheorem{example}[theorem]{Example}

\usepackage{aurical}

\usetikzlibrary{shadows}
\usetikzlibrary{arrows}
\usetikzlibrary{shapes}

\usepackage[numbers]{natbib}

\usepackage[foot]{amsaddr}

\begin{document}
\title{KROM: Kernelized Reduced Order Modeling}
\author{Aras Bacho$^1$, Jonghyeon Lee$^1$, Houman Owhadi$^1$}
\date{\today}
\address{$^1$ Department of Computing and Mathematical Sciences, Caltech, CA, USA.}

\email{bacho@caltech.edu}
\email{jlee9@caltech.edu}
\email{owhadi@caltech.edu}

\begin{abstract}
We propose KROM, a kernel-based reduced-order framework for fast solution of nonlinear partial differential equations. KROM formulates PDE solution as a minimum-norm (Gaussian-process) recovery problem in an RKHS, and accelerates the resulting kernel solves by sparsifying the precision matrix via sparse Cholesky factorization. A central ingredient is an empirical kernel constructed from a snapshot library of PDE solutions (generated under varying forcings, initial data, boundary data, or parameters). This snapshot-driven kernel adapts to problem-specific structure—boundary behavior, oscillations, nonsmooth features, linear constraints, conservation and dissipation laws—thereby reducing the dependence on hand-tuned stationary kernels. The resulting method yields an implicit reduced model: after sparsification, only a localized subset of effective degrees of freedom is used online. We report numerical results for semilinear elliptic equations, discontinuous-coefficient Darcy flow, viscous Burgers, Allen--Cahn, and two-dimensional Navier--Stokes, showing that empirical kernels can match or outperform Mat\'ern baselines, especially in nonsmooth regimes. We also provide error bounds that separate discretization effects, snapshot-space approximation error, and sparse-Cholesky approximation error.
\end{abstract}

\maketitle

\section{Introduction}

 In this paper, we introduce KROM, a novel kernel- and Gaussian-process-based reduced order model for solving high-dimensional problems. Gaussian processes have been used for PDE solvers \cite{chen2021solving,
fang2024, mora2025, bach25} as well as for computational (hyper)graph completion \cite{houman_cgc}, scientific discovery \cite{bourdais2024}, time series analysis \cite{rubio2007, pandey2023, peters2009}, and uncertainty quantification \cite{bajgiran2021uncertainty,tavallali2024,jinghaili2005,yikianli2021,shuaiguo2021}. Their appeal lies in its interpretability, accuracy, flexibility, calibrated uncertainty estimates, and established theoretical guarantees. KROM exploits a sparse approximation of the Cholesky factor \cite{yifan2025, schaefer2021, schaefer2024, huan2025} of the inverse kernel matrix to significantly reduce the time required to solve the PDE and combines it with a particular choice of kernel that is constructed by summing snapshots of solutions of a PDE under different forcing terms, initial conditions or boundary conditions. We also highlight the primary challenges encountered when applying ROMs to these PDEs.


ROMs have historically been motivated by problems in aerodynamics, where potentially millions of grid points and time steps are required to solve the Navier-Stokes equations under turbulent flows accurately \cite{noack2005, rowley2017, brunton2020, barnett2022quadratic_manifold, barnett2023nn_augmented_rom, grimberg2020_stability_prom, peherstorfer2016opinf, swischuk2019nonintrusive,mohan2018lstrom,wang2022fno_les_iso,luo2024fno_rt_les}. In structural mechanics \cite{structural2012,guo2018gaussian, touze2021review_nonlin_mor, Bac2023, park2023gpr_ice_struct}, mathematical models must accurately capture the effects of stress and vibrations over long time frames; ROMs enable engineers to quickly simulate the time evolution of these physical properties and evaluate the structural stability of the object with minimal loss of accuracy. ROMs also play an important role in climate, weather and ocean models, which take into account many atmospheric, oceanic, and land phenomena at once \cite{maulik2022,GWSM4C, oceans2009, oceans2006}; a ROM can consider only the most important patterns to yield predictions of future weather and climate at many locations simultaneously. Further applications of ROMs have arisen in other fields as diverse as biomedical engineering \cite{chellappa2024,lassila2013,buoso2019, ballarin2016, cicci2023} and electromagnetics \cite{CAD2024, li2021nonintrusive_em_spline, yu2024nimor_em}.


\subsection{Contributions and structure of the paper}
The main contributions of our work are: 

\begin{itemize}
    \item \textbf{Kernelized implicit ROM for nonlinear PDEs.}
    We formulate PDE solution as a Gaussian-process optimal recovery problem in an RKHS and obtain an \emph{implicit} reduced-order model: the solution is represented in the span of a small, adaptively selected subset of degrees of freedom induced by sparsification of the precision matrix.

    \item \textbf{Empirical (snapshot) kernel that encodes solution structure.}
    We introduce an empirical kernel built from a library of PDE solution snapshots so that the kernel captures problem-specific features (boundary behavior, oscillations, nonsmoothness, linear constraints, and conservation/dissipation structure), reducing or eliminating hyperparameter tuning and mitigating the smoothness bias of generic kernels (e.g.\ Matérn).

    \item \textbf{Scalable sparse-precision implementation.}
    Using sparse Cholesky factorization of the inverse kernel matrix, we obtain near-linear (up to polylog factors) memory/time complexity in the number of collocation constraints, enabling fast online solves after offline snapshot/kernel construction.

    \item \textbf{Non-intrusive reduced modeling.}
    KROM does not require Galerkin projection or intrusive reduced operators; nonlinearities are handled through the constrained GP recovery (solved by Gauss--Newton), while sparsification provides the computational reduction.

    \item \textbf{A quantitative error budget.}
    We provide convergence/error bounds that separate (i) collocation/discretization error (fill distance), (ii) snapshot/modeling error (finite empirical kernel space / manifold approximation), and (iii) sparse-Cholesky approximation error (controlled by the sparsity radius $\rho$), yielding guidance for balancing $M$, $N$, and $\rho$.

    \item \textbf{Performance on nonsmooth and multiscale regimes.}
    Through extensive experiments (semilinear elliptic, discontinuous-coefficient Darcy, viscous Burgers, Allen--Cahn, and 2D Navier--Stokes), we demonstrate that KROM matches or outperforms Matérn-based GP solvers and classical methods, especially in problems with discontinuities or sharp gradients.
\end{itemize}

\subsection{Outline}
Section~2 presents the kernel-based optimal-recovery formulation and motivates empirical (snapshot) kernels. Section~3 reports numerical experiments across several nonlinear PDE benchmarks. Section~4 provides error bounds and an explicit decomposition into discretization, snapshot-space approximation, and sparse-Cholesky effects. Section~5 concludes with limitations and directions for future work.


\subsection{Literature Review}

Several well-known classes of ROMs, both classical projection methods and more recent machine learning models, have been developed for solving nonlinear PDEs. We detail the most common techniques below; a more comprehensive review can be found in \cite{hinze2021model,quarteroni2016reduced, kuanlu2021}.

Proper orthogonal decomposition (POD) \cite{lumley1967structure,Sirovich1987} is widely considered the first ROM. Originally developed for analyzing coherent structures of linear PDEs but later extended to turbulent flows \cite{aubry1988, POD1993}, the basic idea behind POD is to project the solution into an orthonormal basis of functions. Formally, we would assume that the solution $u(\mathbf{x},t)$ of a space-time PDE is approximately of the form:

\begin{equation}
    u(\mathbf{x},t) \approx \sum^N_{i=1} a_i (t) \phi_i (\mathbf{x}),
\end{equation}

where the $\phi_i$ functions form a basis in the solution space of the PDE. The key step is to collate the solution snapshots into a matrix, perform a singular value decomposition (SVD), and choose the first $r \ll N$ columns of the left SVD matrix $\Phi$ to be the orthogonal basis functions $\phi_i, \; i= 1,\dots, r$. The purpose behind this low-rank SVD is to select the modes which capture most of the energy of the solution snapshots in an $L^2$ sense. Obtaining these orthonormal functions allows us to reduce the nonlinear PDE into a system of low-dimensional ODEs, which are faster to solve than the original PDE. Variants of POD remain common in applications such as aerodynamics \cite{simpson2024vprom,amsallem2008interpolation,li2022efficient,SOLANFUSTERO2022111672} and oceanography \cite{Kitsios2024}.

However, a naive projection to the lower-dimensional space is computationally expensive for nonlinear ODEs due to the need to recompute the nonlinear differential operator at every time step; this has inspired the development of refined algorithms that are more efficient than classical POD. One such method, gappy POD \cite{EversonSirovich1995,dechanubeksa2020application}, was originally developed as a way to reconstruct PDE solution data from partial observations. In gappy POD, a binary measurement matrix $P \in \mathbb{R}^{m \times N}$, where  $m \ll N$ is the number of randomly selected points at which we observe the PDE solution, is applied to the vector of snapshots $u$. We then seek a vector $a$ which minimizes the least-squares problem $||Pu -P \Phi a||^2$. This procedure is particularly efficient in real-life applications in which measurement sensors may not be present at all points in space or time. Later methods such as least-squares Petrov-Galerkin (LSPG) \cite{carlsberg2011} and Gauss-Newton for approximated tensors (GNAT) \cite{carlsberg2011} use a slight variation of gappy POD for nonlinear term hyperreduction in which the nonlinear term of the PDE instead of the solution itself was approximated from sparse measurements using the same procedure as gappy POD, before being plugged back into the system of ODEs derived with POD.

Missing point estimation \cite{astrid2004, astrid2008} improves gappy POD by using a greedy algorithm to take representative snapshots of the data, either from the solution itself, in which case it is known as principled sparse sampling \cite{Willcox2006, Yildirim2009Efficient}, or the nonlinear term of the PDE, that captures the dynamics of the system instead of choosing sampling points randomly. The points are typically selected by taking a SVD ($p \ll N$) of the snapshot matrix $\mathbf{F}$ and performing a pivoted-QR decomposition of the left SVD matrix $\Phi_F$ to identify the $p \ll N$ most informative rows and selecting the corresponding collocation points.
We solve the same least-squares reconstruction problem as in gappy POD to recover the missing data.

Among classical ROMs, arguably the most well-studied improvement on POD is the discrete empirical interpolation method (DEIM) \cite{chaturantabut2010}. DEIM works by performing a low-rank SVD on the snapshot of nonlinear terms, much like GNAT and LSPG, but assumes access to snapshots at every point. The aim of DEIM is to interpolate the nonlinearity $\mathcal{N}(u)$ using only $p \ll N$ points. The point selection process works as follows: we take the first column $\xi_1$ of $\mathbf{\Xi}$, the left SVD matrix of $\mathbf{F}$, and select the position of its maximum value as the index of the first collocation point. Next, we compute a residual vector formed when the $j$th basis vector is projected onto the previously chosen basis vectors, and choose the the $j$th collocation point by selecting the index where the residual vector is largest until we have reached $j=p$. The main achievement of DEIM is that it created a ROM in which the nonlinear computations are independent of the size of the original problem. Q-DEIM \cite{DrmacGugercin2016} simplifies the point selection process by performing a pivoted-QR decomposition on the nonlinear snapshot matrix and selecting the points corresponding to the position of the $1$ entry in the first $p$ columns of the pivot matrix. Further extensions of DEIM such as TQ-DEIM \cite{TQDEIM} and S-DEIM \cite{farazmand2024} have been developed to deal with problems involving tensor-valued and sparse data respectively.




In recent years, Dynamic Mode Decomposition (DMD) \cite{dmdintro}, a data-driven method that decomposes the system dynamics into modes associated with specific temporal frequencies, has been successfully applied to linear systems and extended to handle nonlinear PDEs by combining DMD with machine learning approaches. The key idea behind Nonlinear DMD (nDMD) \cite{HUHN2023111852,nayak,kalur2023data,libero2024extended} is to extract nonlinear modes by using a combination of data-driven methods and neural networks or other machine learning techniques. Nonlinear DMD aims to approximate the underlying nonlinear dynamics of the system, extracting modes that evolve with both time and parameter changes.

In addition, a large number of machine learning ROMs have proliferated in recent years. These typically rely on neural networks \cite{LI2025113705,chen2023crom,wen2023reduced} to approximate the solution manifold. Physics-informed neural networks (PINNs) \cite{CHEN2021110666, hijazi2023pod,CHEN2024117198}, for example, embed the governing equations directly into the loss function, allowing the model to learn the solution in the reduced space while satisfying the PDE constraints. Additionally, autoencoders \cite{romor2023non, conti2023reduced} and variational autoencoders \cite{simpson2024vprom} have been used to compress high-dimensional data and discover low-dimensional representations of the solution space. One of the main challenges is the need for large datasets to train the models; this can be computationally expensive, and ensuring that they generalize well to new problems is still an open question.

\section{Solving PDEs with kernels and sparse Cholesky}
We recapitulate the key ideas behind solving PDEs with kernels via minimum-norm recovery in an RKHS, and we motivate \emph{empirical (snapshot) kernels} as a data-driven alternative to generic stationary kernels. The guiding principle is that the kernel is not merely a numerical device: it encodes the notion of similarity and complexity that the solver uses. Empirical kernels learn this notion of similarity from representative solutions, thereby adapting to anisotropy, nonstationarity, and low-dimensional structure that frequently arise in PDE solution families. Finally, we briefly explain why sparse Cholesky factorizations are central for scalability.

\subsection{Kernel minimum-norm recovery for PDE constraints}

Let $\Omega\subset \mathbb{R}^n$ be a Lipschitz domain and $T>0$. Consider the general initial--boundary value problem (IBVP)
\begin{align}\label{eq:ibvp_main}
\begin{cases}
\mathcal{P}(t,\mathbf{x},u)=f, & \mathbf{x}\in\Omega,\ t\in(0,T),\\
\mathcal{B}(t,\mathbf{x},u)=g, & \mathbf{x}\in\partial\Omega,\ t\in(0,T),\\
\mathcal{I}(\mathbf{x},u)=h, & \mathbf{x}\in\Omega,
\end{cases}
\end{align}
where $\mathcal{P}$ denotes a (possibly nonlinear) differential operator, $\mathcal{B}$ a boundary operator, and $\mathcal{I}$ an initial-condition operator.

The classical Gaussian process (GP) methodology to solve \eqref{eq:ibvp_main} for given $f,g,$ and $h$ \cite{chen2021solving} aims to minimize the reproducing kernel Hilbert space (RKHS) norm induced by a kernel $K$, subject to enforcing \eqref{eq:ibvp_main} at a finite set of collocation points. More formally, let $\tilde{\Omega}:=[0,T]\times\Omega$ and write $z=(t,\mathbf{x})\in\tilde{\Omega}$. Choose a positive semidefinite kernel $K$ on $\tilde{\Omega}\times\tilde{\Omega}$ with associated RKHS $\mathcal{H}_K$ and norm $\|\cdot\|_K$. Given collocation points
\[
(z_i)_{i=1}^{M_\Omega}\subset (0,T)\times\Omega,\qquad
(\tilde z_j)_{j=1}^{M_{\partial \Omega}}\subset (0,T)\times\partial\Omega,\qquad
(\mathbf{x}_k)_{k=1}^{L}\subset \Omega,
\]
the kernel method seeks the minimum-norm function that satisfies the PDE constraints at these locations:
\begin{align}\label{eq:min_norm_main}
\min_{v\in \mathcal{H}_K}\ \|v\|_K^2
\quad \text{s.t.}\quad
\begin{cases}
\mathcal{P}(z_i,v)=f(z_i), & i=1,\dots,M_\Omega,\\
\mathcal{B}(\tilde z_j,v)=g(\tilde z_j), & j=1,\dots,M_{\partial \Omega},\\
\mathcal{I}(\mathbf{x}_k,v)=h(\mathbf{x}_k), & k=1,\dots,L.
\end{cases}
\end{align}
The corresponding finite-dimensional problem is, by the representer Theorem~\ref{th:representer.thm}, given by
\begin{align} \label{eq.min_fin}
    \min_{\overset{\mathbf{z} \in \mathbb{R}^{D M_\Omega + M_{\partial\Omega}+L}}{\textrm{s.t.} 
    F(\mathbf{z}) = \mathbf{y},
    }} \mathbf{z}^T K(\phi,\phi)^{-1} \mathbf{z},
\end{align} 
where $\phi$ is the vector of evaluation and differential functionals applied to $u$ at the collocation points; $F$ encodes the operators $\mathcal{P}, \mathcal{B}, \mathcal{I}$ evaluated at those points; $D$ is the number of distinct linear differential functionals used in the interior constraints; and $\mathbf{y}$ stacks the corresponding right-hand sides. When $\mathcal{P}$ is nonlinear, the constraints are nonlinear in $v$ and \eqref{eq.min_fin} is typically solved by Gauss--Newton or related sequential linearization schemes; for linear operators it reduces to a linear system.

\subsubsection{Representer form.} By the representer theorem, see Theorem \eqref{th:representer.thm}, the minimizer (when it exists) lies in the span of kernel sections associated with the collocation sites. In its simplest form, the solution can be written as
\begin{align}\label{eq:representer_main}
u(z)=\sum_{m=1}^{N_{\text{data}}} c_m\, K(z,z_m),
\qquad N_{\text{data}}:=DM_\Omega+M_{\partial\Omega}+L,
\end{align}
with coefficients $c_m$ chosen so that the constraints in \eqref{eq:min_norm_main} are satisfied. (For interior PDE constraints involving derivatives, the corresponding representer involves kernel sections acted on by the appropriate differential/evaluation functionals).

\subsubsection{Why the kernel choice matters}

The kernel $K$ fixes the RKHS $\mathcal{H}_K$ and therefore the notion of ``simplicity'' promoted by the minimum-norm principle. Standard stationary kernels such as Mat\'ern or Gaussian kernel primarily encode generic smoothness and isotropy. However, for many PDEs the family of admissible solutions is not well described by global isotropic regularity: it may exhibit transport-aligned features, localized layers, sharp transitions, or coherent structures with a few degrees of freedom. In such cases, a generic smoothness prior can be statistically inefficient as it requires many constraints, and numerically delicate.

This motivates \emph{problem-adapted kernels} that reflect the geometry of the solution set, rather than imposing a generic notion of smoothness.

\subsection{Empirical kernels: learning similarity from data}\label{subsec:empirical_kernel_main}

Assume we have access to representative solutions (snapshots) $\{u_i\}_{i=1}^N$, obtained by varying forcing terms, boundary conditions, initial conditions, and/or parameters while keeping the domain and the operators fixed. We define the empirical kernel
\begin{align}\label{eq:empirical_kernel_main}
K_N(z,z'):=\frac{1}{N}\sum_{i=1}^N u_i(z)\,u_i(z'),
\qquad z,z'\in\tilde{\Omega}.
\end{align}
This kernel is automatically positive semidefinite (it is an average of rank-one kernels) and therefore induces a valid RKHS.

The definition \eqref{eq:empirical_kernel_main} can be motivated from several complementary viewpoints; together they explain why $K_N$ is a natural substitute for hand-chosen stationary kernels.

\subsubsection*{Learning the right notion of similarity (geometry of the solution set)}
A kernel is a similarity measure on $\tilde\Omega$, and the induced RKHS norm determines which functions are considered ``simple.'' The empirical kernel declares two points $z$ and $z'$ to be similar if their values co-vary across realistic solutions. This automatically produces a \emph{nonstationary}, \emph{anisotropic} kernel aligned with the observed solution manifold, capturing transport directions, localized activity, and coherent patterns without prescribing them a priori.

\subsubsection*{Kernel as feature map (linear kernel in snapshot coordinates)}
Define the feature map $\Phi:\tilde{\Omega}\to\mathbb{R}^N$ by
\[
\Phi(z):=\frac{1}{\sqrt{N}}\big(u_1(z),\dots,u_N(z)\big).
\]
Then
\[
K_N(z,z')=\langle \Phi(z),\Phi(z')\rangle_{\mathbb{R}^N}.
\]
Thus $K_N$ is a \emph{linear kernel} in a data-driven feature space whose coordinates are physically meaningful (snapshot values). Rather than hand-designing basis functions, we let the PDE generate features that reflect the relevant dynamics.

\subsubsection*{Reduced-order modeling in kernel form (ROM/POD bridge)}
Classical reduced-order modeling approximates new solutions in a low-dimensional space identified from training solutions (POD or reduced bases). The empirical kernel is the kernelized packaging of the same idea: the hypothesis space induced by $K_N$ is the snapshot span, while the kernel machinery provides a principled way to fit coefficients by enforcing PDE constraints, without explicitly assembling projection-based reduced operators.

\subsubsection*{Operator viewpoint: approximating the solution map without inverting it}
Write the PDE solution map abstractly as $u=\mathcal{S}(f)$ (linear case: $\mathcal{S}=L^{-1}$). If snapshots arise as $u_i=\mathcal{S}(f_i)$ under a family of excitations $\{f_i\}$, then $K_N$ estimates the output correlation structure of $\mathcal{S}$ under those excitations:
\[
K_N(z,z')\approx \mathbb{E}\big[u(z)u(z')\big]
= \mathbb{E}\big[\mathcal{S}(f)(z)\,\mathcal{S}(f)(z')\big].
\]
In the linear setting, for sufficiently rich (``white-ish'') forcing ensembles, this covariance behaves like a regularized inverse-type object (up to the forcing covariance). In this way, $K_N$ acts as a data-driven surrogate for the smoothing/anisotropy of the inverse operator, without requiring eigenfunctions, Green's functions, or explicit inversion.

\subsubsection*{Probabilistic viewpoint: empirical covariance and optimal linear prediction}
If the snapshots are interpreted as samples from the distribution of solutions relevant to the task (random forcing/parameters), then $K_N$ is precisely the empirical covariance of the random field $u(z)$. Under the corresponding GP prior, the minimum-norm interpolant (equivalently, the noiseless GP posterior mean) is the optimal linear predictor. This provides a clear statistical rationale: we tailor the prior to the task distribution instead of imposing generic smoothness.

\subsubsection*{Practical payoffs: efficiency, constraints, compression}
Empirical kernels offer several practical advantages:
\begin{itemize}
\item \emph{Sample efficiency:} aligning the kernel with the solution manifold can reduce the number of collocation points needed to resolve layers/transport compared to stationary smoothness kernels.
\item \emph{Built-in structure:} if all snapshots satisfy a homogeneous linear constraint (e.g.\ homogeneous boundary conditions), then any linear combination does as well, so such constraints can be satisfied by construction.
\item \emph{Compression:} $K_N$ is finite rank ($\le N$) and can be compressed via POD/PCA by retaining only dominant modes, yielding a controllable accuracy--cost trade-off.
\end{itemize}

\subsubsection{RKHS induced by the empirical kernel and representer form}
The empirical kernel induces an RKHS that coincides with the snapshot span (with a data-dependent inner product). Concretely, the representer theorem with $K_N$ yields
\begin{align}\label{eq:representer_empirical_main}
u(\cdot)=\sum_{m=1}^{N_{\text{data}}}\alpha_m\,K_N(\cdot,z_m)
=\frac{1}{N}\sum_{i=1}^N \beta_i\,u_i(\cdot),
\end{align}
for suitable coefficients $\alpha_m$ (and induced snapshot coefficients $\beta_i$ determined by the constraints). Hence the recovered solution is a linear combination of snapshots, with coefficients chosen to enforce \eqref{eq:ibvp_main} at the collocation sites. This makes explicit the reduced-basis inductive bias of empirical kernels.

Hence $K_N$ is a positive finite-rank operator whose range equals $\mathrm{span}\{u_i\}$. In particular, when the snapshots are $L^2$-orthonormal, $K_N$ is the orthogonal projector onto $\mathrm{span}\{u_i\}$. Consequently, the associated RKHS is exactly the snapshot space (with a data-dependent inner product):
\[
\mathcal{H}_{K_N}=\mathrm{span}\{u_1,\dots,u_N\}.
\]
Thus, choosing $K_N$ enforces a reduced-basis inductive bias: the recovered solution is sought in the span of previously observed solutions. In order to have the optimal  approximation of the solution manifold by a linear subspace spanned by the snapshots, one can obtain the snapshots by suitable algorithms such as the greedy algorithm, see Section \ref{se.greedy}. As discussed in Section \ref{se.greedy}, this can guarantee a (sub)-exponential decay of the Kolmogorov $N$-width in dependence on the number of snapshot solutions. 

\subsubsection{Homogeneous linear constraints satisfied by construction}
Suppose $\mathcal{B}$ is linear and all snapshots satisfy a homogeneous constraint such as
\begin{align}\label{eq:hom_constraint_main}
\mathcal{B}(t,\mathbf{x},u)=0,\qquad (t,\mathbf{x})\in (0,T)\times\partial\Omega.
\end{align}
Then any linear combination $\sum_i \beta_i u_i$ satisfies \eqref{eq:hom_constraint_main}. Consequently, the recovered solution \eqref{eq:representer_empirical_main} satisfies that constraint identically, and the corresponding boundary conditions need not be enforced explicitly in the collocation system. For inhomogeneous constraints ($\mathcal{B}u=g\neq 0$) one may use standard remedies such as a lifting $u=u_0+v$ with $\mathcal{B}u_0=g$ and $\mathcal{B}v=0$, or augment the snapshot set.

\subsubsection{Low-rank compression via POD/PCA}
Since $K_N$ is finite rank, it can be compressed by retaining only dominant POD/PCA modes. If $\{\phi_\ell\}_{\ell=1}^r$ are leading modes with energies $\sigma_1^2\ge\cdots\ge\sigma_r^2$ (with $r\ll N$), a truncated kernel
\[
K_r(z,z'):=\sum_{\ell=1}^r \sigma_\ell^2\,\phi_\ell(z)\phi_\ell(z')
\]
captures the dominant variability while reducing computational cost; the truncation error is controlled by the neglected energy $\sum_{\ell>r}\sigma_\ell^2$.

\subsubsection{Low-rank compression via POD/PCA and connection to POD ROM}
Since $K_N$ is finite rank, it can be compressed by retaining only dominant POD/PCA modes.
Let $\{\phi_\ell\}_{\ell\ge 1}$ be the POD modes of the snapshot ensemble with energies
$\sigma_1^2\ge \sigma_2^2\ge \cdots$ (defined with respect to the inner product used for POD,
e.g.\ the $L^2$ inner product or a finite-element mass inner product).
Then the empirical covariance kernel admits the spectral expansion
\[
K_N(z,z')=\sum_{\ell=1}^{\operatorname{rank}(K_N)} \sigma_\ell^2\,\phi_\ell(z)\phi_\ell(z').
\]
For $r\ll N$, define the truncated (rank-$r$) kernel
\[
K_r(z,z'):=\sum_{\ell=1}^r \sigma_\ell^2\,\phi_\ell(z)\phi_\ell(z'),
\]
which captures the dominant variability while reducing computational cost; the truncation error
is controlled (in the POD sense) by the neglected energy $\sum_{\ell>r}\sigma_\ell^2$.

We note that the associated RKHS is exactly the POD trial space
\[
\mathcal{H}_{K_r}=\mathrm{span}\{\phi_1,\dots,\phi_r\}=:V_r,
\]
and any $v\in \mathcal{H}_{K_r}$ admits the POD expansion $v=\sum_{\ell=1}^r a_\ell \phi_\ell$. Thus, choosing $K_r$ makes KROM a reduced-order method with POD trial space $V_r$.

We further note, that KROM coincides with the standard POD--Galerkin ROM if we choose the constraints in the optimal recovery problem \eqref{eq:min_norm_main} to be the weak formulation of the PDE of interest, i.e.,  
\[
F_j(v):=\langle \mathcal{R}(u_r), \phi_j\rangle = 0,\qquad j=1,\dots,r.
\]

\subsection{Computational aspect: why sparse Cholesky is relevant}

The kernel systems arising from \eqref{eq:min_norm_main} are typically dense if formed naively. Scalability therefore hinges on exploiting structure. Empirical kernels are particularly amenable to efficient linear algebra because they are low rank (or well-approximated by low rank after POD/PCA). In addition, many PDE constraint assemblies and kernel approximations admit \emph{effective sparsity} through localization/screening ideas (dominant near-field interactions) or sparse-plus-low-rank decompositions. Once a sparse (or sparsified) positive definite system is obtained, sparse Cholesky factorizations provide an efficient and numerically stable way to solve the linear systems that arise in the linear case and within each Gauss--Newton step for nonlinear problems.

In summary, kernel minimum-norm recovery provides a flexible framework for enforcing PDE constraints, while empirical kernels replace generic smoothness priors by a similarity notion learned from representative solutions, yielding reduced-basis behavior, improved adaptation to complex solution manifolds, and computational opportunities through low-rank structure and sparse factorizations.

\section{Numerical Examples}
In this section, we perform extensive numerical experiments with a variety of examples showing the advantage of the empirical kernel over the Matérn-2.5 kernel which is commonly used for solving PDEs with kernels \cite{yifan2025}. We offer a summary of the setup of each experiment in Table \ref{table:krom_pde_params} in Section 3.6.
\subsection{Semilinear Elliptic Equation}  
\label{sec:semilinear_elliptic}

In this example, we consider the following two-dimensional semilinear elliptic equation in supplemented with inhomogeneous Dirichlet boundary conditions:
\begin{align}
    \begin{cases}
        -\Delta u(\mathbf{x}) + \tau(u(\mathbf{x})) = f(\mathbf{x}), \quad \mathbf{x} \in \Omega,\\
        u(\mathbf{x}) = g(\mathbf{x}), \quad \mathbf{x} \in \partial \Omega,
    \end{cases}
\label{nonlinear_elliptic}
\end{align}
where $f: \Omega \to \mathbb{R}$ denotes the forcing term, $\tau: \mathbb{R} \to \mathbb{R}$ is a nonlinear function (e.g., $\tau(u) = u^3$ or $\tau(u) = \sin(u)$), $g: \partial \Omega \to \mathbb{R}$ is the boundary function.

Nonlinearity in $\tau(u)$ requires iterative solvers such as Newton's method or Picard iteration. 
We randomly generate functions $f_i$ from a centered Gaussian distribution with a Gaussian kernel with lengthscale $\sigma=0.15$, and solve the corresponding PDE with a finite element method combined with Newton's method to construct our empirical kernel.

We use a manufactured-solution setup: we prescribe
\[
u^\star(x_1,x_2)=0.5 \sin(\pi x_1)\sin(\pi x_2)+\sin(2\pi x_1)\sin(2\pi x_2),
\]
set $g=u^\star|_{\partial\Omega}$ (here $g\equiv 0$), and define
$f := -\Delta u^\star + (u^\star)^3$ so that $u^\star$ satisfies \eqref{nonlinear_elliptic} with $\tau(u)=u^3$. Then, for $f$ and $g$, we aim to approximate the true solution with our method using the empirical kernel  on the square grid $\Omega = (0,1)^2$ under different choices of sparsity parameter $\rho$, number of solutions $N$ used to construct the kernel and the number of collocation points $M$.

We use three Gauss-Newton iterations to solve the PDE with both the empirical kernel and the Matérn-2.5 kernel 

\begin{equation}
    K_\theta (x,y) =  \left(1+ \sqrt{5} \frac{||x-y||}{\theta} + \frac{5}{3} \frac{||x-y||^2}{\theta^2}\right) \exp \left(\frac{-\sqrt{5}||x-y||}{\theta}\right)
\end{equation}

with lengthscale parameter $\theta =0.3$, where we have applied the sparse Cholesky algorithm to the inverse of the empirical and Matérn kernel matrices. The results in Figures \ref{fig:elliptic_4} show an initial rapid decay in the relative error between the true solution and the kernel solution as $N$, the number of solutions used to construct the empirical kernel, increases while we fix $M=32^2$ and $\rho = 4$, before plateauing to values between $0.01$ and $0.02$ when $N=60$, as the error cannot decrease any further without increasing the number of collocation points. Similarly, the relative error decreases as $\rho$ increases while $M=32^2$ and $N = 200$ stay fixed, as displayed in Figure \ref{fig:elliptic_5}; this is expected as a larger sparsity parameter means that the sparse Cholesky factorization is a better approximator of the true precision matrix. We also observe in Figure \ref{fig:elliptic_6} a near-linear decrease in the log of the relative error of the numerical solution generated by both kernels as $\log(M)$ increases while the number of solutions $N=5000$ and sparsity parameter $\rho=4$ is kept fixed. In this example, both the Matérn-2.5 and empirical kernels accurately recover the true solution because it is a smooth function.

\begin{figure}[tbh]
\centering
\subfloat{\label{fig:elliptic_1}
\includegraphics[width=.31\textwidth]{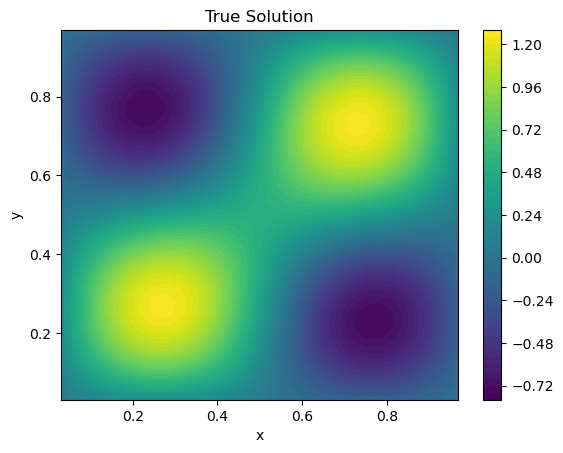}
}
\hfill
\subfloat{\label{fig:elliptic_2}
\includegraphics[width=.31\textwidth]{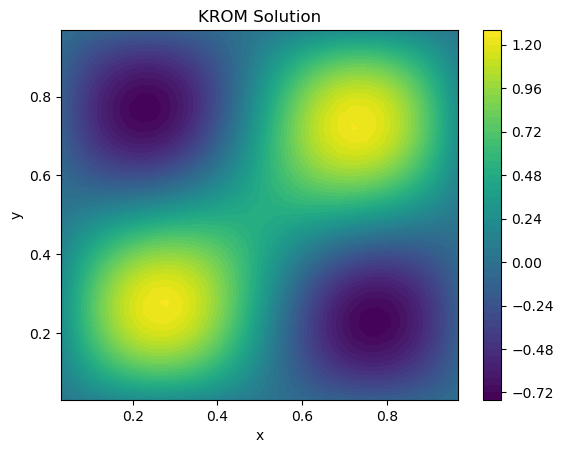}
}
\hfill
\subfloat{\label{fig:elliptic_3}
\includegraphics[width=.31\textwidth]{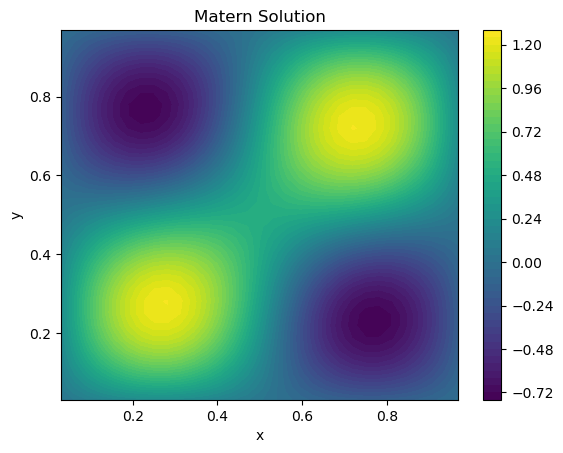}
}
\hfill
\caption{True solution (left) compared with the solution given by the empirical kernel (centre) and Matérn-2.5 kernel (right) for $M=64^2$ collocation points and sparsity parameter $\rho = 4$, as well as $N=5000$ solutions for constructing the empirical kernel.}
\label{nonlin}
\end{figure}

\begin{figure}[tbh]
\centering
\subfloat{\label{fig:elliptic_4}
\includegraphics[width=.31\textwidth]{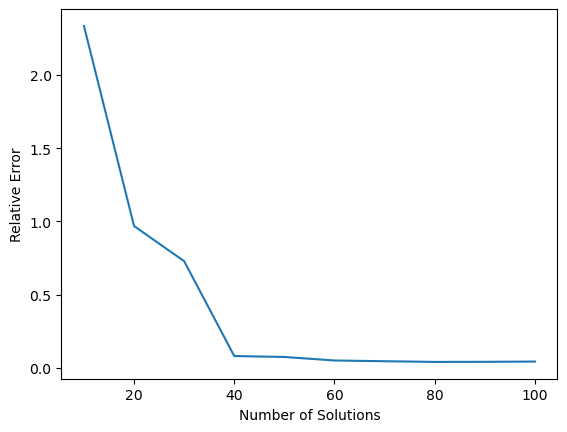}
}
\hfill
\subfloat{\label{fig:elliptic_5}
\includegraphics[width=.31\textwidth]{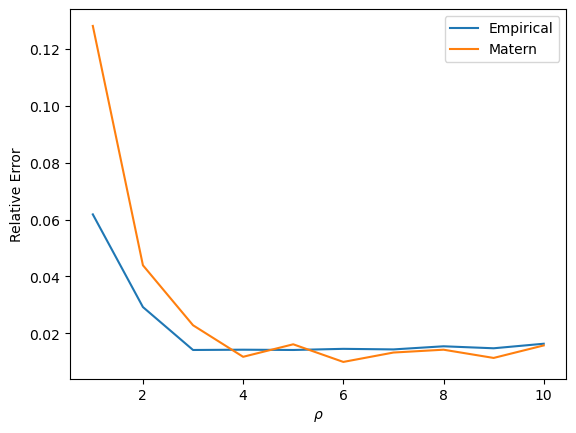}
}
\hfill
\subfloat{\label{fig:elliptic_6}
\includegraphics[width=.31\textwidth]{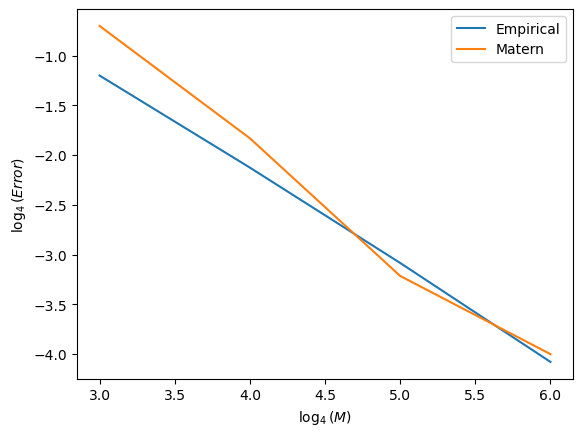}
}
\hfill
\caption{Decrease in relative error of solution obtained by the empirical kernel as the number of solutions $N$ (left) increases while number of collocation points $M=32^2$ and sparsity parameter $\rho = 5$ are fixed; as $\rho$ (centre) increases while $N=200, M=32^2$ are fixed; and as $M$ (right) increases for $\rho=4, N=5000$ are fixed.}
\label{nonlin_error}
\end{figure}

\begin{figure}[tbh]
\centering
\subfloat{\label{fig:elliptic_7}
\includegraphics[width=.45\textwidth]{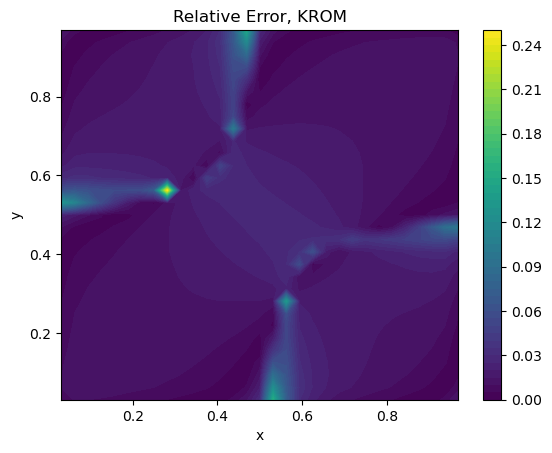}
}
\hfill
\subfloat{\label{fig:elliptic_8}
\includegraphics[width=.45\textwidth]{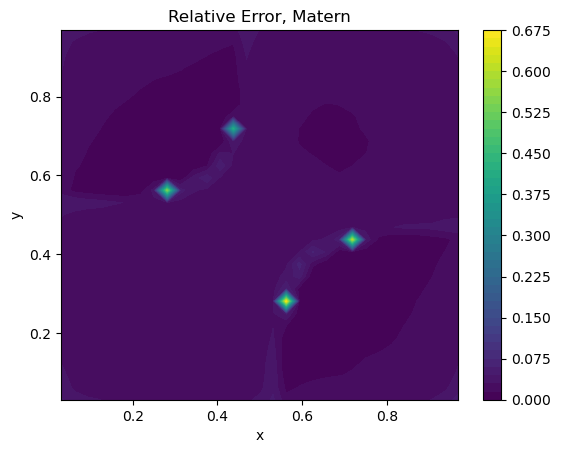}
}
\hfill
\caption{Pointwise relative error of the solution to the nonlinear elliptic PDE obtained by the empirical kernel for $M=32^2, \rho = 5$, $N=200$ for the empirical (left) and Matérn-2.5 (right) kernels; notice the relative error is higher where the true solution is near zero.}
\end{figure}

\subsection{Darcy Flow} 
\label{sec:darcy} 

Darcy's law governs the flow of a fluid through a porous medium and is described by the following:
\begin{align}
    \begin{cases}
        -\nabla \cdot (k(\mathbf{x}) \nabla u(\mathbf{x})) + \tau(u(\mathbf{x})) = f(\mathbf{x}), \quad \mathbf{x} \in \Omega,\\
        u(\mathbf{x}) = g(\mathbf{x}), \quad \mathbf{x} \in \partial \Omega,
    \end{cases}
\label{eq:darcy}
\end{align}
where $k(\mathbf{x}) > 0$ is the permeability tensor (possibly spatially varying), $u(\mathbf{x})$ is the pressure, $\tau (u)$ is a nonlinear function, $f(\mathbf{x})$ is the source term, and $g(\mathbf{x})$ represents Dirichlet boundary conditions. 

The Darcy Flow PDE is a suitable benchmark for ROMs because spatial variability in $k(\mathbf{x})$ can be extreme, making reduced bases problem-specific and non-trivial to generalize. If $k(\mathbf{x})$ varies by orders of magnitude within the domain, it is difficult to construct a low-dimensional representation capturing fine-scale flow features.

For nonlinear Darcy flow, $k$ may depend on $p$ (e.g., $k(p) = k_0(1 + \alpha p)$), leading to a nonlinear PDE. In our example, we consider $\tau (u) = u^3$ and the following discontinuous coefficient function:

\begin{equation}
   k(\mathbf{x}) =  \frac{101}{2}-\frac{99}{2} (-1)^{\lfloor 8x_1 \rfloor+\lfloor 8x_2 \rfloor}
\end{equation}

As in the previous example (semilinear elliptic equation), we generate our solutions $u_i$ by solving the PDE \ref{eq:darcy} with Dirichlet boundary condition $g=0$ over the domain $\Omega = (0,1)^2$ and $M=32^2$ collocation points, with a finite element method coupled with Newton's method; we have generated each $f_i$ using a centered Gaussian process using the Gaussian kernel with lengthscale $\sigma=0.2$ as our covariance function. We then test our KROM with a different right-hand side, namely $f=1$. In Figures \ref{darcy} and \ref{darcy_error}, the benefits of using the empirical kernel become transparent: since $k(\mathbf{x})$ is not continuous, the solution $u$ is not smooth and thus will not lie in the RKHS induced by the Matérn-2.5 kernel, hence the relative error will be high regardless of the choice of sparsity parameter $\rho$. On the other hand, $u$ can be recovered accurately using the empirical kernel to solve the PDE, since the empirical kernel is an inner product of snapshots of the solutions $u_i$, each of which lies in the same solution space $H^1_0 (\Omega)$ as $u$.

\begin{figure}[tbh]
\centering
\subfloat{\label{fig:darcy_1}
\includegraphics[width=.31\textwidth]{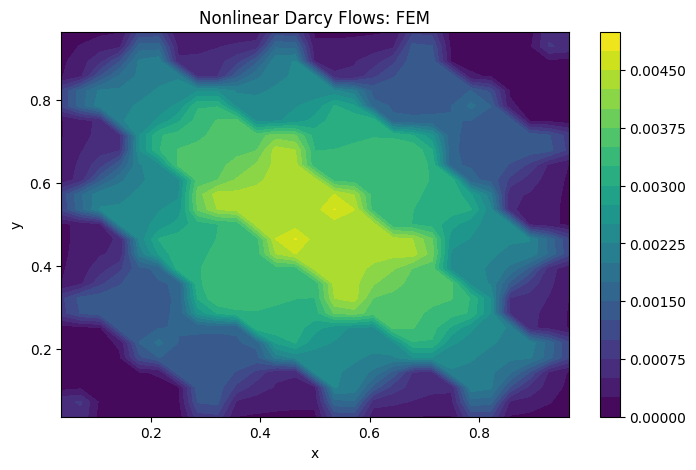}
 }
\hfill
\subfloat{\label{fig:darcy_2}
\includegraphics[width=.31\textwidth]{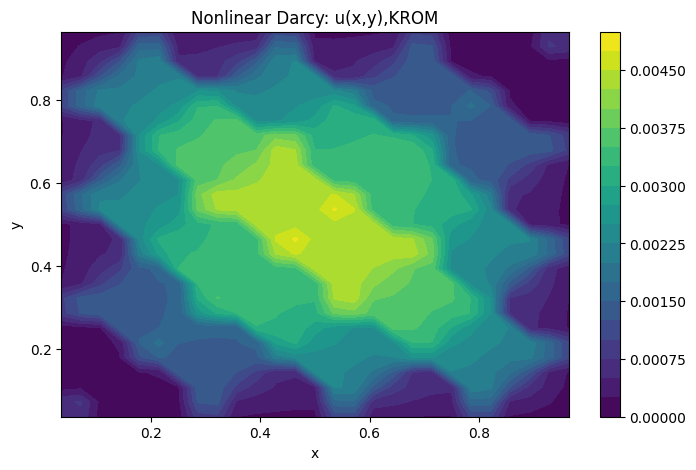}
}
\hfill
\subfloat{\label{fig:darcy_3}
\includegraphics[width=.31\textwidth]{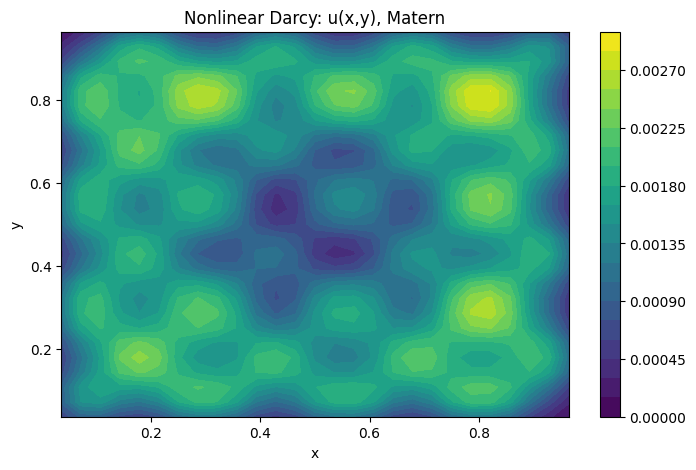}
}
\hfill
\caption{The Darcy Flow numerical solution obtained by a finite element method (left), an empirical kernel with $N=40$ solutions (centre) and a Matérn kernel with lengthscale parameter $\theta =0.3$ (right). We have used sparsity parameter $\rho = 4, M=32^2$ collocation points and 2 Gauss-Newton iterations to solve the PDE for both kernels.}
\label{darcy}
\end{figure}

\begin{figure}[tbh]
\centering
\subfloat{\label{fig:darcy_4}
\includegraphics[width=.45\textwidth]{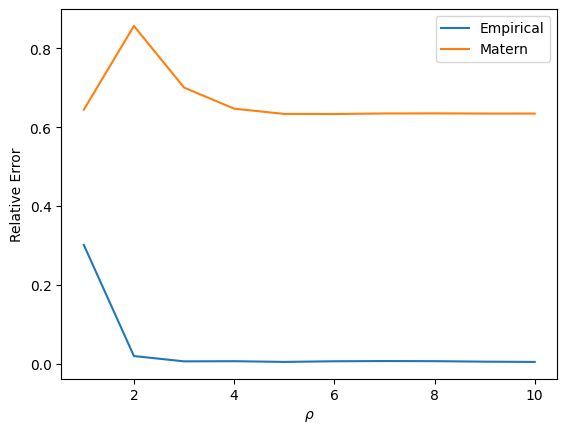}
}
\hfill
\subfloat{\label{fig:darcy_5}
\includegraphics[width=.45\textwidth]{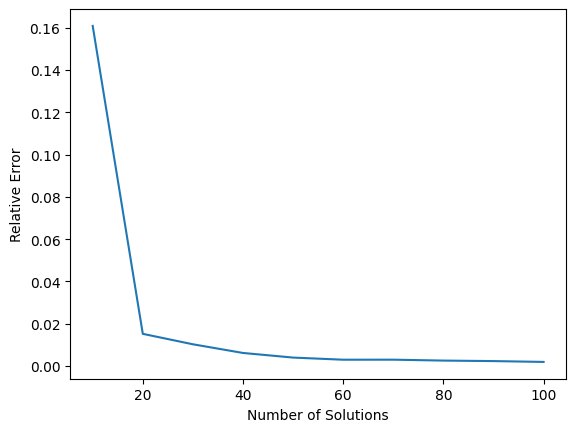}
}
\hfill
\caption{Relative error of the solution to the Darcy Flow PDE obtained by the empirical kernel for $M=32^2$ collocation points as sparsity parameter $\rho$ increases while the number of solutions used to construct the empirical kernel $N=200$ stays fixed (left); and as $N$ increases for fixed $\rho=4$ (right).}
\label{darcy_error}
\end{figure}

\begin{figure}[tbh]
\centering
\subfloat{\label{fig:darcy_6}
\includegraphics[width=.45\textwidth]{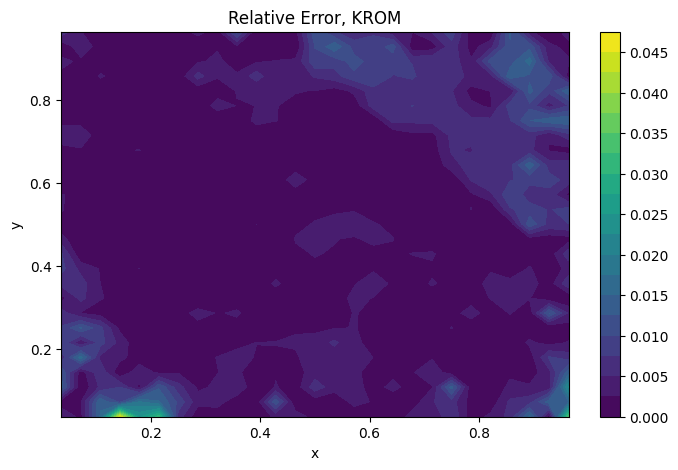}
}
\hfill
\subfloat{\label{fig:darcy_7}
\includegraphics[width=.45\textwidth]{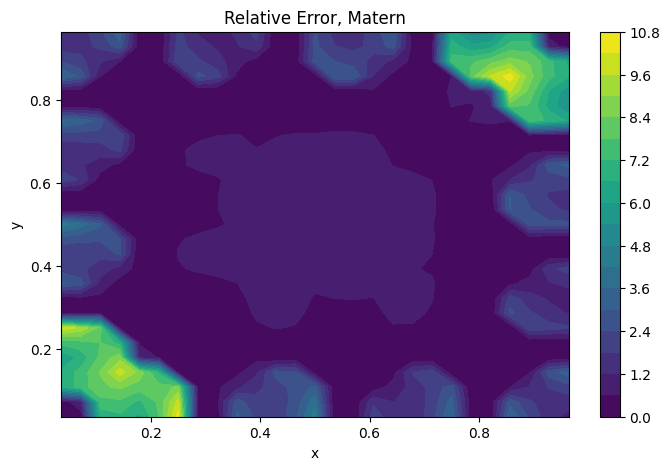}
}
\hfill
\caption{Pointwise relative error of the solution to the Darcy Flow PDE obtained by solving the PDE for $M=32^2, \rho = 4$ with the empirical (left) and Matérn-2.5 (right) kernels. We have used $N=40$ solutions to construct the empirical kernel to generate this error plot.}
\label{darcy_pointwise}
\end{figure}

\subsection{Burgers' Equation}
\label{sec:burgers}

A common testbed for nonlinear advection-diffusion problems, the one-dimensional viscous Burgers' equation is:
\begin{align}\label{burgers}
    \begin{cases}
         \frac{\partial u}{\partial t}(t,x) + u(t,x) \frac{\partial u}{\partial x}(t,x)
  = \nu \frac{\partial^2 u}{\partial x^2}(t,x), \quad (t,x) \in (0,T)\times \Omega,\\
  u(0,x)=u_0, \quad x \in \Omega,\\
        u(t,x) = g(x), \quad (t,x) \in [0,T]\times\partial \Omega,
    \end{cases}
\end{align}
 where
\begin{itemize}
  \item $u(x,t)$ is a scalar field (e.g., velocity),
  \item $\nu$ is the viscosity (often small in test problems).
\end{itemize}

The main challenges for solving the Burgers equation are that in the limit of small $\nu$, shocks (very steep gradients) can develop, challenging low-dimensional approximations. Furthermore, like many PDEs, the nonlinear advection term can lead to instability if not carefully handled in a reduced setting.

Consider \eqref{burgers} with $\Omega=[-1,1]$, viscosity parameter $\nu = 0.001$, initial condition $u(x,0)=-\sin(\pi x)$ and boundary conditions $u(-1,t)=0=u(1,t)$. To generate solutions for the empirical kernel, we consider $8$ initial conditions of the form $\sum^{10}_{i=1} a_i (\cos(b_i \pi x) + \sin(b_i \pi x))$, where the coefficients $a_i$ are drawn from a standard normal distribution and the $b_i$ terms are generated from the discrete uniform distribution $U\{1,2\}$, and solve the corresponding Burgers equation with a fifth-order WENO scheme from time $t=0$ to $t=1$ with 2000 points in space in the interval $[-1,1]$ and a time step of $h=10^{-3}$. We exploit the translational invariance in $x$ of solutions of the Burgers equation by shifting each of these periodic solutions by $x=-0.8,-0.6,\dots,0.6,0.8$; this yields us a total of 80 solutions. Once we construct the empirical kernel $K(x,y) = \frac{1}{N_T} \frac{1}{N} \sum^N_{i=1} \sum^T_{j=1} u_i (x,t_j) u_i (y,t_j)$, we perform a sparse Cholesky decomposition of the precision matrix $K^{-1}$ as in the previous examples. We use the following Crank-Nicolson discretization with a step size of $\Delta t=0.04$ to obtain a differential equation for time $t_{n+1}$, which we forward solve in time with the empirical and Matérn-2.5 kernels with two Gauss-Newton iterations at each time $t_n$:

\begin{equation}
\frac{\hat{u} (x,t_{n+1})-\hat{u} (x,t_n)}{\Delta t} +\frac{1}{2} (\hat{u} (x,t_{n+1}) \partial_x \hat{u} (x,t_{n+1})+\hat{u} (x,t_n) \partial_x \hat{u} (x,t_n)) = \frac{\nu}{2} (\partial^2_x \hat{u} (x,t_{n+1}) +\partial^2_x \hat{u} (x, t_n))
\label{burgers_discretized}
\end{equation}

\begin{figure}[tbh]
\centering
\subfloat{\label{fig:burgers_1}
\includegraphics[width=.45\textwidth]{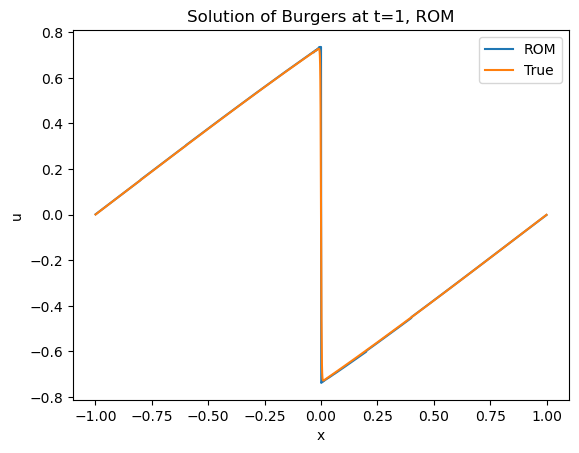}
}
\hfill
\subfloat{\label{fig:burgers_2}
\includegraphics[width=.45\textwidth]{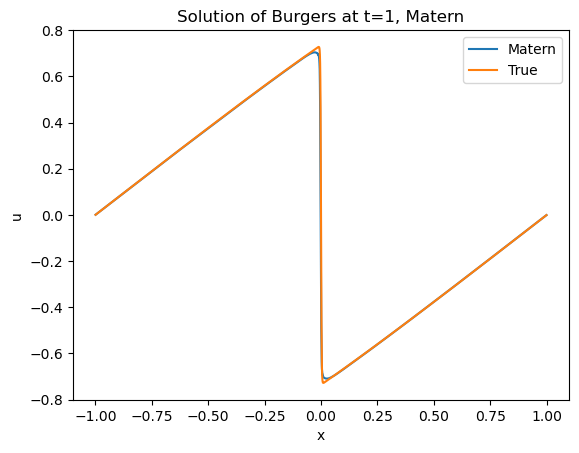}
}
\hfill
\caption{Solving the Burgers equation with the empirical (left) and Matérn-2.5 (right) kernels, sparsity parameter $\rho=5,M=2001$ collocation points and $N=40$ solutions for the empirical kernel.}
\label{burgers_graph}
\end{figure}

A close look at Figure \ref{burgers_graph} shows that while the empirical kernel is able to capture the gradient at the shock at the $t=1$, the Matérn-2.5 kernel with lengthscale $\theta =0.05$ is unable to do the same and instead attempts to 'smooth out' the region where the shock occurs, suggesting that a higher sparsity pattern for the Cholesky factor is required for the Matérn-2.5 kernel to reach the same accuracy as the empirical kernel. This observation is known as the Gibbs phenomenon and is caused by the Gaussian process being unable to explain the jump, spreading the transition between states over a length scale determined by the kernel. On the other hand, the shocks are already present in the solutions $u_i$ which we sum over to create the empirical kernel.

\begin{figure}[tbh]
\centering
\subfloat{\label{fig:burgers_3}
\includegraphics[width=.45\textwidth]{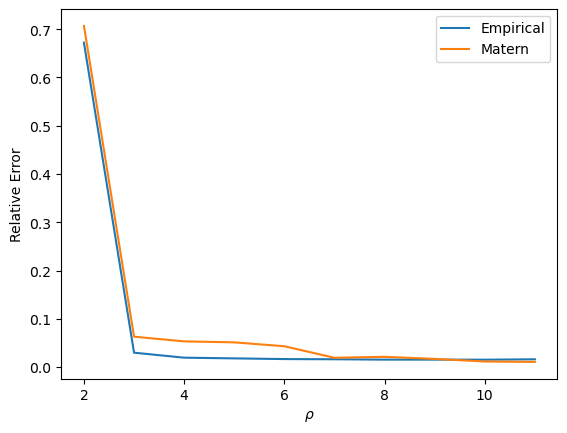}
}
\hfill
\subfloat{\label{fig:burgers_4}
\includegraphics[width=.45\textwidth]{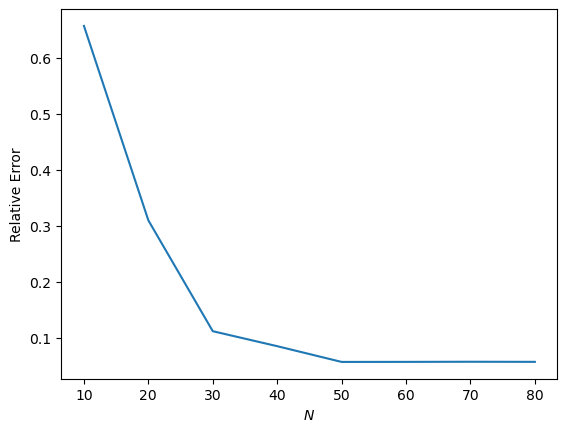}
}
\hfill
\caption{Relative error of the solution to the Burgers equation obtained by the empirical and Matérn-2.5 kernels for $M=2001$ collocation points as sparsity parameter $\rho$ increases for $N=40$ solutions used for the empirical kernel (left); and relative error for the empirical kernel as $N$ increases as $\rho=5$ and $M=2001$ stay fixed (right).}
\label{burgers_error}
\end{figure}

\subsection{Allen--Cahn Equation}
\label{sec:allen_cahn} 
            
The Allen--Cahn equation describes the motion of phase boundaries in a binary system. A common form is:
\begin{align}\label{allen_cahn}
    \begin{cases}
          \frac{\partial u}{\partial t}(t,\mathbf{x}) = \varepsilon^2 \Delta u(t,\mathbf{x}) - F'(u(t,\mathbf{x})), \quad (t,\mathbf{x}) \in (0,T)\times \Omega,\\
  u(0,\mathbf{x})=u_0, \quad \mathbf{x} \in \Omega,\\
        u(t,\mathbf{x}) = g(\mathbf{x}), \quad (t,\mathbf{x}) \in [0,T] \times \partial \Omega,
    \end{cases}
\end{align}
where
\begin{itemize}
  \item $u(\mathbf{x}, t)$ is the order parameter indicating phase states,
  \item $\varepsilon$ is a small parameter related to the interface width,
  \item $F(u)$ is a double-well potential, for example $F(u) = \frac{1}{4}(u^2 - 1)^2$.
\end{itemize}

The Allen-Cahn equation is another common benchmark for ROMs \cite{chaturantabut2010} for several key reasons. First, the sharp transitions in $u$ (the interface between phases) can be difficult to represent with a small number of basis functions. Second, strong nonlinearities in the potential term require careful treatment or hyper-reduction methods. Third, phase-field models often exhibit slow evolution after the initial transient, demanding a robust reduced basis that remains accurate over large time scales.

We consider a 2D version of the Allen-Cahn equation with $F(u) = u^3-u, \varepsilon =0.01$, domain $\Omega = (0,2\pi)^2$ and a homogeneous Dirichlet boundary condition $u(x,t) = 0$ for $x \in \partial \Omega$. To generate the solutions for the empirical kernel $K((x,y),(x',y')) = \frac{1}{N_T} \frac{1}{N} \sum^N_{i=1} \sum^{N_T}_{j=1} u_i (x,y,t_j) u_i (x',y',t_j) $, we generate $N$ initial conditions of the form $u^0_i (x) = \sum^{5}_{n=1} \sum^5_{m=1} a_{nm} \sin (n x) \sin (m y)$, where each $a_{nm}$ is drawn from a standard normal distribution and $n,m \in \{1,2,3,4,5\}$ are picked randomly from a discrete uniform distribution. We use a simple forward-time central-space (FTCS) scheme to solve \eqref{allen_cahn} for each of these initial conditions with $M=51^2$ points over $\Omega$ and a time step of $\Delta t= 10^{-4}$. As we did for the Burgers equation, we use the Crank-Nicolson scheme to solve a time-independent PDE for the next time step until the final time for both the Matérn-2.5 and empirical kernels to test our reduced-order model:

\begin{equation}
    \frac{u^{n+1}-u^n}{\Delta t} = \varepsilon^2 \frac{\Delta u^{n+1} -\Delta u^{n}}{2} - \frac{F'(u^{n+1})-F'(u^n)}{2}
\end{equation}

In the results below, we have used a large step size of $\Delta t= 0.05$ to solve the Allen-Cahn equation with initial condition $u(x,y,0) = 0.25 \sin(3x) \sin(3y)$ up to $t=5$. The numerical solutions obtained by the empirical kernel and the Matérn-2.5 kernel with lengthscale $\theta = 0.1$ at the end time are displayed alongside the FTCS solution in Figure \ref{ac}. The analysis of the results is similar to that of the Burgers equation in Section 3.3. The Matérn kernel, while resembling the FTCS solution, performs worse than the empirical kernel at the boundaries between the different interfaces in the solution as it attempts to smooth the jump between each interface. These very sharp internal layers are already baked into each of the reference solutions used to construct the empirical kernel, which is why it does not suffer from the same problem to the same extent.

\begin{figure}[tbh]
\centering
\subfloat{\label{fig:ac_1}
\includegraphics[width=.31\textwidth]{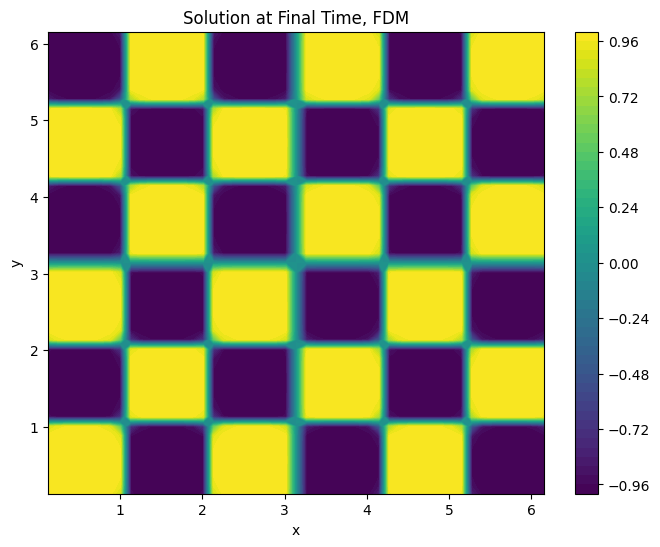}
}
\hfill
\subfloat{\label{fig:ac_2}
\includegraphics[width=.31\textwidth]{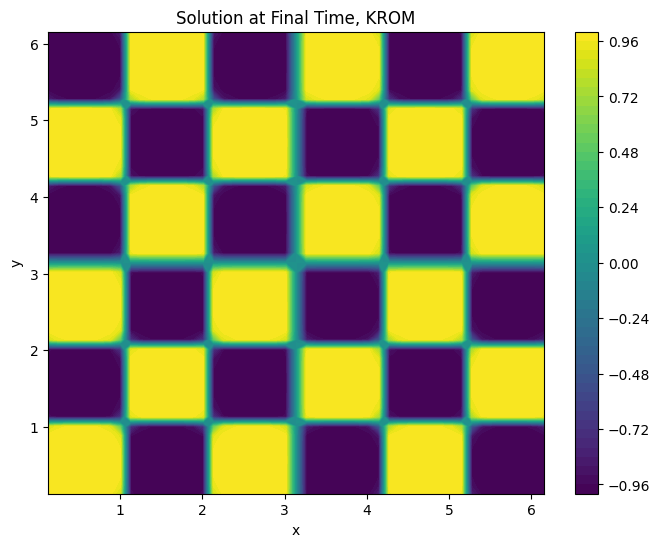}
}
\hfill
\subfloat{\label{fig:ac_3}
\includegraphics[width=.31\textwidth]{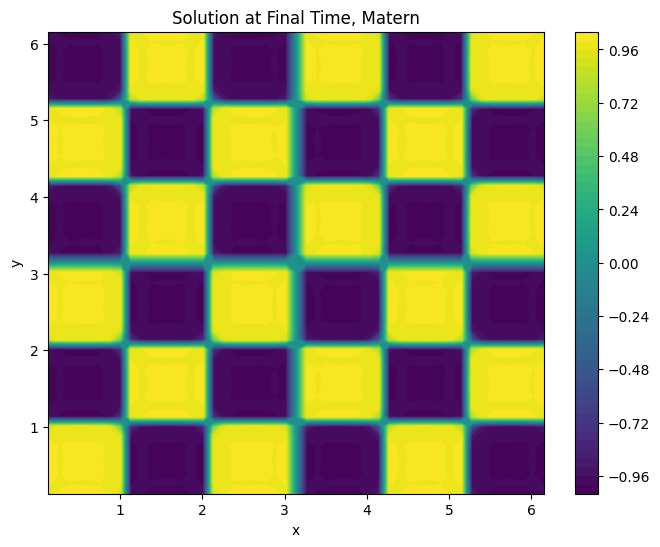}
}
\hfill
\caption{Results for the Allen-Cahn equation for the finite difference method (left), empirical kernel (center) and Matérn-2.5 kernel (right)}
\label{ac}
\end{figure}

\begin{figure}[tbh]
\centering
\subfloat{\label{fig:ac_4}
\includegraphics[width=.45\textwidth]{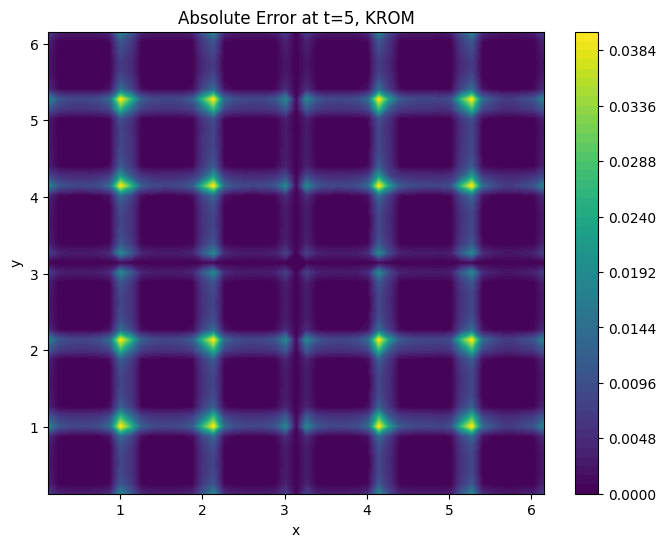}
}
\hfill
\subfloat{\label{fig:ac_5}
\includegraphics[width=.45\textwidth]{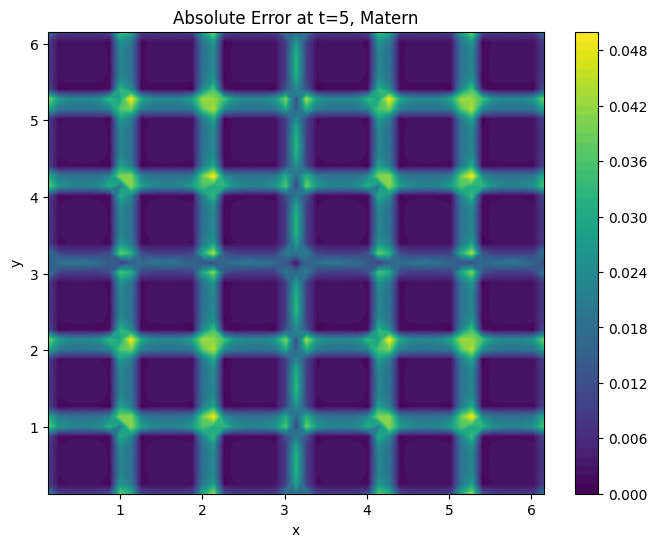}
}
\hfill
\caption{Absolute error of the solution to the Allen-Cahn equation obtained by the empirical and Matérn kernels for $M=51^2$ collocation points, sparsity parameter $\rho=5$ and $N=10$ solutions to generate the empirical kernel; we illustrate the absolute error instead of the relative error as the accuracy at each point can be seen more clearly in the former than in the latter.}
\end{figure}

\begin{figure}[tbh]
\centering
\subfloat{\label{fig:ac_6}
\includegraphics[width=.45\textwidth]{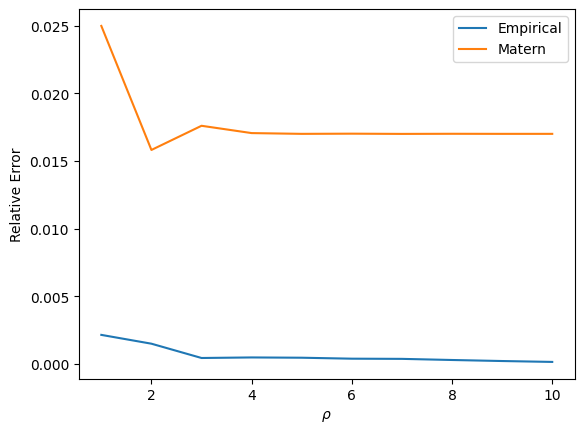}
}
\hfill
\subfloat{\label{fig:ac_7}
\includegraphics[width=.45\textwidth]{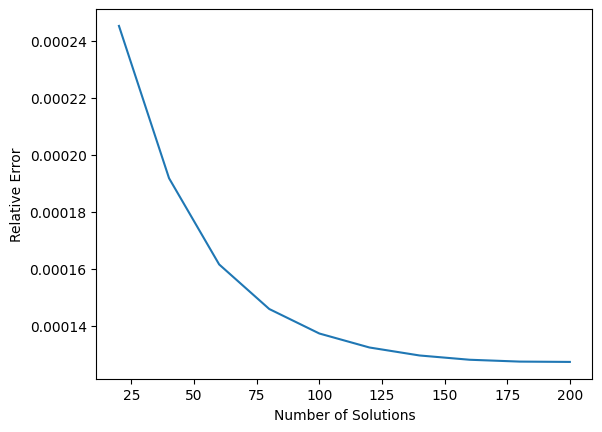}
}
\hfill
\caption{Relative error of the solution to the Allen-Cahn equation obtained by the empirical and Matérn-2.5 kernels as sparsity parameter $\rho$ increases (we have used $N=10$ solutions to create the empirical kernel); and the relative error for the empirical kernel as $N$ increases as $\rho=5$ (right). We have used $M=51^2$ collocation points for all our experiments}
\end{figure}

\subsection{Navier-Stokes Equation}
\label{sec:navier_stokes} 

The incompressible Navier--Stokes equations in two dimensions are given by the following system of equations:
\begin{align}\label{eq:navier-stokes}
    \begin{cases}
           \frac{\partial \mathbf{u}}{\partial t}(t,\mathbf{x}) + (\mathbf{u}(t,\mathbf{x}) \cdot \nabla)\mathbf{u}(t,\mathbf{x})= -\frac{1}{\rho}\nabla p(t,\mathbf{x}) + \nu \nabla^2 \mathbf{u}(t,\mathbf{x}), \quad (t,\mathbf{x}) \in (0,T)\times \Omega,\\
           \nabla \cdot \mathbf{u}(t,\mathbf{x}) = 0\quad (t,\mathbf{x}) \in (0,T)\times \Omega\\
  \mathbf{u}(0,\mathbf{x})=u_0, \quad \mathbf{x} \in \Omega,\\
        \mathbf{u}(t,\mathbf{x}) = \mathbf{g}(\mathbf{x}), \quad (t,\mathbf{x}) \in [0,T]\times \partial \Omega,
    \end{cases}
\end{align}

where 
\begin{itemize}
  \item $\mathbf{u}(x,y,t)$ is the velocity field,
  \item $p(x,y,t)$ is the pressure,
  \item $\rho$ is the density; here $\rho =1$
  \item $\nu$ is the kinematic viscosity; we use $\nu = 10^{-3}$ in our experiments.
\end{itemize}

The Navier-Stokes equation is the most commonly studied problem in fluid mechanics. Much care must be taken when solving Navier-Stokes with ROMs, because projection-based ROMs can introduce instabilities if the chosen reduced subspace and the treatment of nonlinear terms are not sufficiently accurate. Additionally, if the Reynolds number or boundary conditions change, the reduced basis must be robust enough to capture these variations.

Instead of solving for the velocity field $\mathbf{u}$ directly, we solve the equation for the vorticity $\omega$ which is defined by the curl of $\mathbf{u}$:

\begin{equation}
    \partial_t \omega + \mathbf{u} \cdot \nabla \omega = \nu \Delta \omega.
\label{vorticity}
\end{equation}

To recover the velocity from vorticity, we introduce a streamfunction $\psi$ satisfying
\[
-\Delta \psi = \omega,
\qquad
\mathbf u = \nabla^\perp \psi := (\partial_y \psi,\,-\partial_x \psi),
\]
which automatically enforces $\nabla\cdot \mathbf u = 0$ (for periodic domains or with appropriate boundary conditions).

This formulation is advantageous for the 2D Navier-Stokes problem because vorticity is a scalar when the domain is 2D. In order to create our empirical kernel matrix for the vorticity, we generate $N$ random smooth fields using $8$ Fourier modes. Specifically, we take initial conditions of the form

\[
u_0(x,y)
=
\operatorname{Re}\!\left(
\sum_{\substack{|k_1|\le 8 \\ |k_2|\le 8}}
\frac{Z_{k_1,k_2}}{1 + k_1^2 + k_2^2}
\,e^{i\,(k_1 x + k_2 y)}
\right),
\]

with $\mathbf{k}=(k_1,k_2)$ being the wavevector and

\[
Z_{k_1,k_2} = \xi_{k_1,k_2}\,e^{i\phi_{k_1,k_2}},
\qquad
\xi_{k_1,k_2}\sim\mathcal N(0,1),
\quad
\phi_{k_1,k_2}\sim\mathrm{Unif}(0,\pi).
\]

and solve the corresponding vorticity equation with a forward Euler scheme and the Poisson equation for the streamfunction with a fast Fourier transform (FFT). We then test our KROM with another randomly generated initial condition and use step size of $\Delta t =0.01$ for the Crank-Nicolson discretization that allows us to forward solve the PDE in time with kernels.

We plot the results at the final time $t=5$ for the empirical kernel 
\begin{align}
K_N ((x,y),(x',y')) = \frac{1}{N_T} \frac{1}{N} \sum^N_{i=1} \sum^{N_T}_{j=1} u_i (x,y,t_j)u_i (x',y',t_j)
\end{align}
and the Matérn-2.5 kernel with parameter $\theta = 0.3$ and also the spectral energy decay in Figures \ref{ns} and \ref{energy_spectrum}. The relative error for the Matérn-2.5 kernel stays nearly unchanged despite increasing the sparsity parameter $\rho$ because the 2D Navier-Stokes violate a number of fundamental assumptions about the Matérn-2.5 kernel. Firstly, solutions to 2D Navier-Stokes are non-stationary in space whereas the Matérn kernel is stationary. More importantly, the vorticity spectrum $\mathbb{E} [|\hat{\omega} (k,t)|^2]$ is time dependent whereas the Matérn-$\nu$ kernel assumes a spectrum that is proportion to $(1+|k|^2)^{-(\nu + d/2}$ with $(\nu - d/2)=3.5$ for all time. Furthermore, a Matérn kernel assumes one dominant correlation length across space and ignores multiscale behavior such as vorticity filaments and sharp gradients. This is shown in Figure \ref{energy_spectrum}, where the Matérn kernel exhibits a steeper decay in the energy spectrum than the classical solution as the Matérn posterior tends to underrepresent energy and enstrophy at high $k$ and smooth out sharp gradients.

The empirical kernel attempts to avoid these problems by approximating the test solution as a linear span of reference solutions, but while this performs better than the Matérn kernel, the relative error does not monotonically decrease as a function of the number of solutions $N$ in the empirical kernel. This is because the solution manifold and flow map of the 2D Navier-Stokes equation is highly nonlinear and cannot be approximated by an arbitrary linear combination of different solutions. The energy spectrum decay $E(k)$ of the empirical kernel decays slower than the classical solution. One possible explanation is that since we are summing over all the $N$ solutions across all the time steps in the empirical kernel, a lot of high-wavenumber Fourier components present in the snapshot fields are being accumulated and inherited in our empirical kernel solution at each time.

\begin{figure}[tbh]
\centering
\subfloat{\label{fig:ns_1}
\includegraphics[width=.31\textwidth]{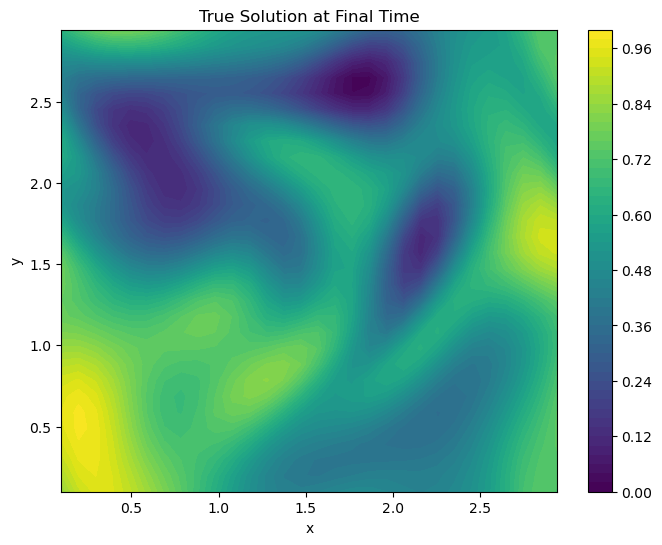}
}
\hfill
\subfloat{\label{fig:ns_2}
\includegraphics[width=.31\textwidth]{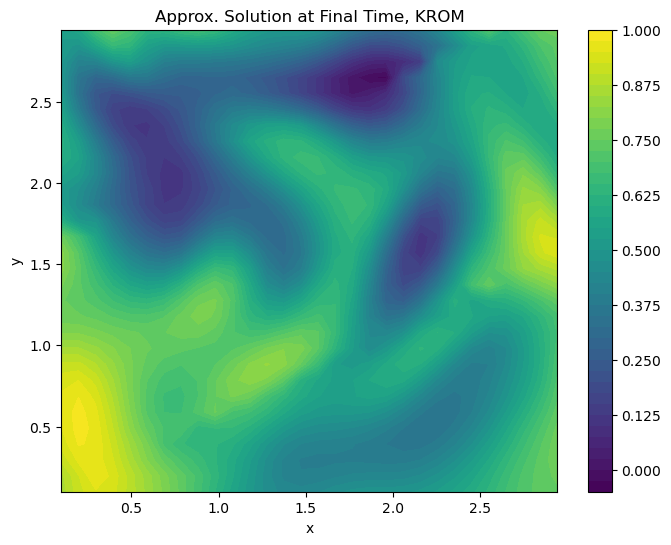}
}
\hfill
\subfloat{\label{fig:ns_3}
\includegraphics[width=.31\textwidth]{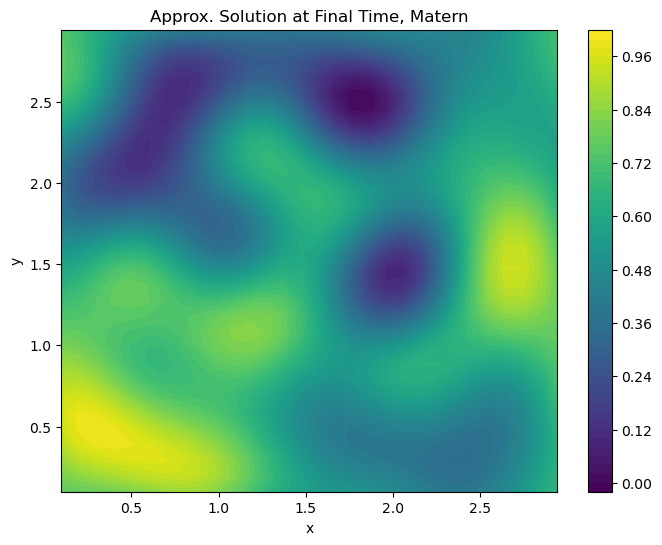}
}
\hfill
\caption{Solution of the 2D Navier-Stokes equation with a classical solver (left), empirical kernel with sparsity parameter $N=50$ solutions (centre) and Matérn-2.5 kernel with lengthscale parameter $\theta=0.3$ (right). We have used $M=32^2$ collocation points and sparsity parameter $\rho=5$ for both kernels.}
\label{ns}
\end{figure}

\begin{figure}[tbh]
\centering
\subfloat{\label{fig:ns_4}
\includegraphics[width=.45\textwidth]{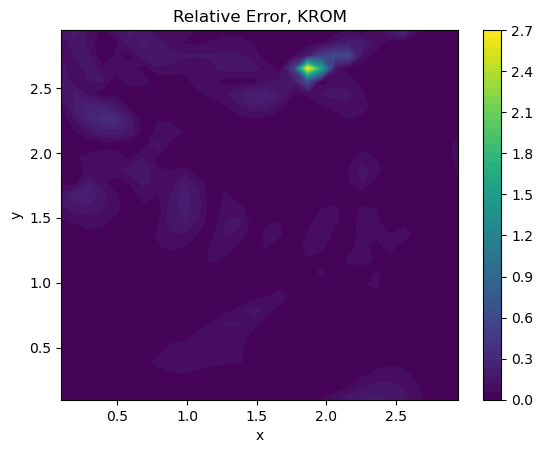}
}
\hfill
\subfloat{\label{fig:ns_5}
\includegraphics[width=.45\textwidth]{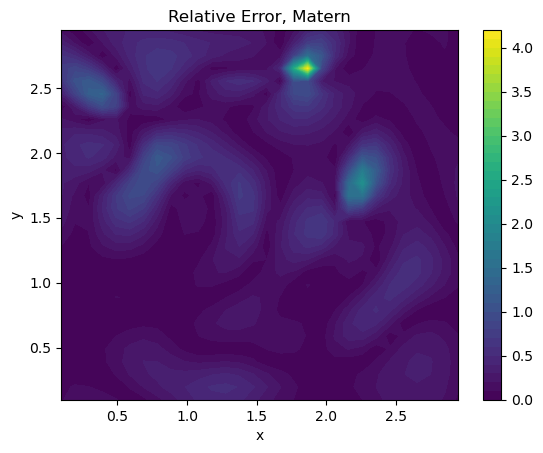}
}
\hfill
\caption{Relative error of the solution to 2D Navier-Stokes obtained by the empirical (left) and Matérn-2.5 (right) kernels for $M=32^2$, $\rho=5, N=100$}
\label{ns_pointwise}
\end{figure}

\begin{figure}[tbh]
\centering
\subfloat{\label{fig:ns_6}
\includegraphics[width=.45\textwidth]{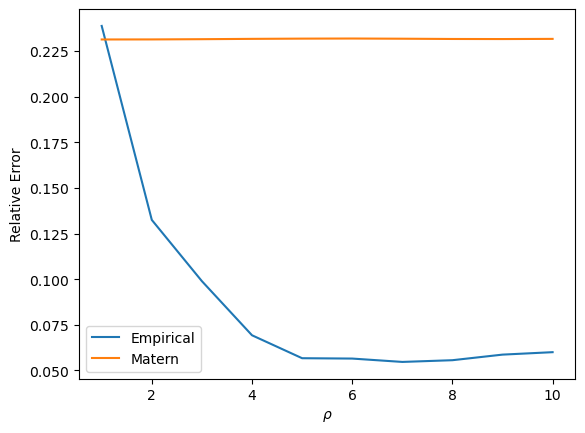}
}
\hfill
\subfloat{\label{fig:ns_7}
\includegraphics[width=.45\textwidth]{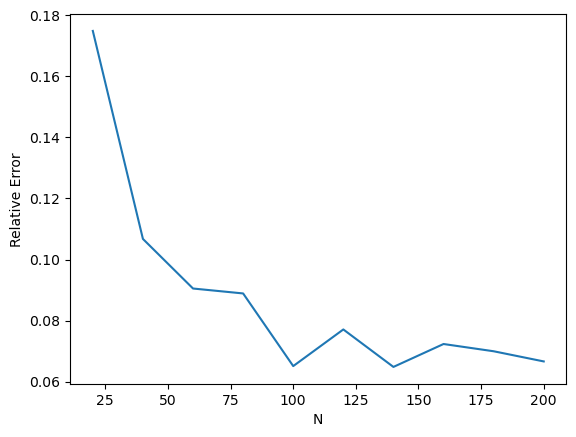}
}
\hfill
\caption{Relative error of the solution to the 2D Navier-Stokes obtained by the empirical kernel for $M=32^2$ as $\rho$ increases while $N=100$ (left) and as $N$ increases as $\rho=5$ (right).}
\label{ns_error}
\end{figure}

\begin{figure}[tbh]
\centering
\label{fig:energy_cascade}
  \includegraphics[width=.51\textwidth]{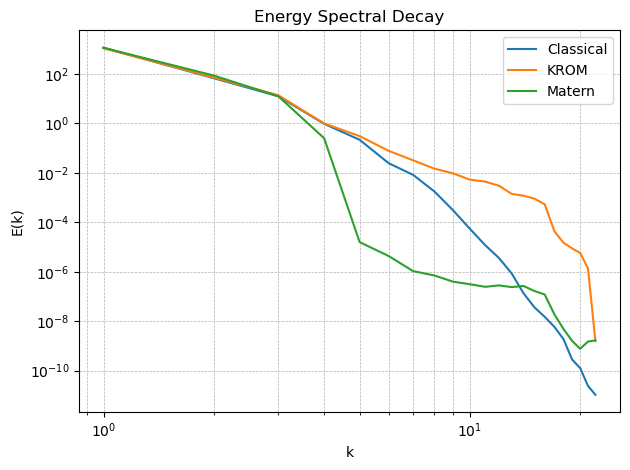}
\hfill
\caption{Energy spectrum decay of the Navier-Stokes equation at final time $t=5$ with different solvers against $k$, the magnitude of the wavevector.} 
\label{energy_spectrum}
\end{figure}

\begin{table}[ht!] \label{table.summary}
    \centering
    \renewcommand{\arraystretch}{1.1}
    \resizebox{\columnwidth}{!}{
    \begin{tabular}{r | c c c}
        experiment set 
            & Semilinear elliptic 
            & Darcy flow 
            & Burgers' 1D \\
        \hline
        PDE type 
            & stationary, semilinear 
            & stationary, nonlinear 
            & time-dependent, advection--diffusion \\
        domain $\Omega$ 
            & $(0,1)^2$ 
            & $(0,1)^2$ 
            & $[-1,1]$ \\
        BC / IC 
            & $u=g = 0 $ on $\partial\Omega$ 
            & $p=g = 0$ on $\partial\Omega$ 
            & $u(\pm1,t)=0$, $u(x,0)=-\sin(\pi x)$ \\
        data distribution 
            & $f$ from GP, $\sigma=0.15$ 
            & $f$ from GP, $\sigma=0.2$ 
            & random trigonometric ICs \\
        collocation grid $M$ 
            & $M = 32^2$ 
            & $M = 32^2$ 
            & $M = 2001$ \\
        snapshots $N$ 
            & $N = 10,20,\dots,200$ FEM solutions 
            & $N = 10,20,\dots,200$ FEM solutions 
            & $N=10,20,\dots,80$ WENO solutions \\
        sparse parameter $\rho$ 
            & $\rho = 4$ 
            & $\rho = 4$ 
            & $\rho = 5$ \\
        kernels 
            & empirical,\; Matérn--2.5 $(\theta=0.3)$ 
            & empirical,\; Matérn--2.5 $(\theta=0.3)$ 
            & empirical,\; Matérn--2.5 $(\theta=0.05)$ \\
        time discretization 
            & Gauss--Newton (3 steps) 
            & Gauss--Newton (2 steps) 
            & Crank--Nicolson, $\Delta t = 0.04$ \\
        solver details 
            & GN in RKHS 
            & GN in RKHS 
            & GN in RKHS \\
        Section 
            & \Cref{sec:semilinear_elliptic} 
            & \Cref{sec:darcy} 
            & \Cref{sec:burgers} \\
        \hline
        experiment set 
            & Allen--Cahn 
            & Navier--Stokes 2D \\
        \hline
        PDE type 
            & time-dependent, phase field 
            & incompressible flow (vorticity) \\
        domain $\Omega$ 
            & $(0,2\pi)^2$ 
            & periodic $(0,2\pi)^2$ \\ 
        BC / IC 
            & $u=0$ on $\partial\Omega$, $u(x,0)=0.25\sin(3x_1)\sin(3x_2)$ 
            & periodic in space, prescribed vorticity IC \\
        data distribution 
            & random trigonometric ICs 
            & band-limited, low-pass–filtered random Fourier series \\
        collocation grid $M$ 
            & $M = 51^2$ 
            & $M = 32^2$ \\
        snapshots $N$ 
            & $N = 10,\,20,\dots,200$ FTCS solutions 
            & $N = 10,\,20,\dots,200$ vorticity fields \\
        sparse parameter $\rho$ 
            & $\rho = 5$ 
            & $\rho = 5$ \\
        kernels 
            & empirical,\; Matérn--2.5 $(\theta=0.1)$ 
            & empirical,\; Matérn--2.5 $(\theta=0.3)$ \\
        time discretization 
            & Crank--Nicolson, $\Delta t = 0.05$ 
            & Crank--Nicolson, $\Delta t = 0.01$ \\
        solver details 
            & GN in RKHS 
            & GN in RKHS, FFT Poisson solve \\
        Section 
            & \Cref{sec:allen_cahn} 
            & \Cref{sec:navier_stokes} \\

    \end{tabular}
    }
    \caption{Summary of numerical experiments for KROM. 
    The row \emph{data distribution} describes how right-hand sides or initial conditions are sampled to generate snapshot solutions used to build the empirical kernel. 
    $M$ denotes the number of collocation points in the kernel discretization, $N$ the number of snapshot solutions, and $\rho$ the sparsity parameter in the sparse Cholesky approximation of $K(\phi,\phi)^{-1}$. 
    Matérn--2.5 kernels with the indicated lengthscale $\theta$ serve as baseline GP solvers.}
    \label{table:krom_pde_params}
\end{table}

\section{Theoretical Analysis}
In this section, we provide an error analysis of KROM in terms of the sparsity parameter $\rho$, the fill distance $h$, and the number of solution snapshots to construct the empirical kernel.

\subsection{Error bounds for the empirical kernel}
In this section, we want to show that the choice of an empirical kernel leads to error bounds in terms of the fill distance. More precisely, we show that if the fill distance goes to zero, i.e., the number of collocation points goes to infinity, the optimal recovery solution converges to the true solution. There holds

\begin{theorem}[Empirical-kernel error estimate and convergence in $N$]\label{thm:empirical_kernel_rate}
Assume the hypotheses of Theorem~1.2 (in particular, the well-posedness/stability assumptions for the IBVP operator and the sampling inequality assumptions that yield an $h^\gamma$--rate). 
Let $U_N$ be the RKHS induced by the empirical kernel constructed from snapshots $\{u_i\}_{i=1}^N$, so that $U_N=\mathrm{span}\{u_1,\dots,u_N\}$ as a set.
Let $u^\star$ denote the unique (true) solution of \eqref{eq:ibvp_main}. 

Let $X_{\Omega_T}\subset [0,T]\times\Omega$ be a set of $M$ interior collocation points (and similarly include boundary/initial collocation sets if present in \eqref{eq:min_norm_main}); define the spatial fill distance
\[
h_{\Omega} := \sup_{x_0 \in \Omega} \inf_{x \in X_{\Omega}} \|x_0 - x\| .
\]
Let $u_{N,h}^\dagger\in U_N$ be a minimizer of the RKHS optimal recovery problem \eqref{eq:min_norm_main} posed over $U_N$ with collocation set(s) of fill distance $h_\Omega$.

Fix $\ell\ge 0$ and assume $s>\ell+\tfrac12$, and let $\gamma:= s-\ell>0$. Then there exist constants $h_0>0$ and $C>0$ such that for all $h_\Omega<h_0$ we have the \emph{two-term} estimate
\begin{equation}\label{eq:empirical_two_term}
\|u_{N,h}^\dagger-u^\star\|_{H^\ell(\Omega)}
\;\le\;
C\, h_\Omega^{\gamma}\,\|u_N^\star\|_{U_N}
\;+\;
\inf_{v\in U_N}\|v-u^\star\|_{H^\ell(\Omega)},
\end{equation}
where $u_N^\star$ is any element of $U_N$ satisfying the continuous constraints of \eqref{eq:ibvp_main} (in particular, if $u^\star\in U_N$ then one can take $u_N^\star=u^\star$), and the constant $C$ is independent of $h_\Omega$ and the minimizer. 

Moreover, suppose in addition that:
\begin{enumerate}
\item\label{ass:uniform_embedding}
there exists a constant $C_{\mathrm{emb}}>0$, independent of $N$, such that
\[
\|v\|_{H^s(\Omega)}\le C_{\mathrm{emb}}\|v\|_{U_N}\qquad\forall v\in U_N,
\]
\item\label{ass:density}
the snapshot spaces approximate the truth in $H^\ell$ in the sense that
\[
\inf_{v\in U_N}\|v-u^\star\|_{H^\ell(\Omega)} \xrightarrow[N\to\infty]{} 0,
\]
\item\label{ass:bounded_comparator}
there exists a sequence $\{\tilde v_N\}_{N\ge 1}$ with $\tilde v_N\in U_N$ and $\sup_N\|\tilde v_N\|_{U_N}<\infty$ such that $\|\tilde v_N-u^\star\|_{H^\ell(\Omega)}\to 0$,
\end{enumerate}
then $\|u_N^\star\|_{U_N}$ is bounded uniformly in $N$ and, for any sequence $h_\Omega=h_\Omega(N)\downarrow 0$ satisfying
\[
h_\Omega(N)^{\gamma}\,\|u_N^\star\|_{U_N}\xrightarrow[N\to\infty]{}0,
\]
we have
\[
\|u_{N,h_\Omega(N)}^\dagger-u^\star\|_{H^\ell(\Omega)}\xrightarrow[N\to\infty]{}0.
\]
In the special case $u^\star\in U_N$ for all sufficiently large $N$, the bias term vanishes and \eqref{eq:empirical_two_term} reduces to the single-term bound
\begin{equation}\label{eq:empirical_single_term}
\|u_{N,h}^\dagger-u^\star\|_{H^\ell(\Omega)}
\;\le\;
C\, h_\Omega^{\gamma}\,\|u^\star\|_{U_N}.
\end{equation}
\end{theorem}

\begin{proof}
    The proof will be presented in Appendix \ref{sec:appendix_proofs}.
\end{proof}

\subsubsection{Interpreting the snapshot-space bias term via Kolmogorov $N$-widths.}
The approximation term 
\[
\inf_{v\in\mathcal H_N}\|v-u^\star\|_{H^\ell(\Omega)},
\]
is a standard reduced-order modeling quantity: it measures how well the snapshot space captures the new solution. If $\mathcal M:=\{u(\mu):\mu\in\mathcal P\}$ denotes the solution manifold over a parameter/forcing set $\mathcal P$, then the best achievable worst-case error over all $N$-dimensional subspaces is the Kolmogorov $N$-width
\[
d_N(\mathcal M;H^\ell):=\inf_{\substack{W\subset H^\ell(\Omega)\\ \dim W=N}}\ \sup_{u\in\mathcal M}\ \inf_{v\in W}\|u-v\|_{H^\ell(\Omega)}.
\]
Thus, for a well-chosen snapshot space $\mathcal H_N$, the bias term can be viewed as an instance of this $N$-width decay. 

\subsection{Kolmogorov \texorpdfstring{$N$}{N}-widths} \label{se.greedy}

The numerical results show that the relative error decreases with an approximately exponential decay in the number of snapshots $N$ we use to construct our empirical kernel. Interestingly, we observe that (sub)-exponential decay despite the fact that we randomly generate those solutions. In this section, we want to show that if we chose those solutions that aim to approximate the solution manifold by an greedy algorithm, the relative error decreases sub-exponentially in $N$. First, we introduce the notion of the Kolmogorov $N$-width with which we measure how well the solution manifold can be approximated by a finite dimensional linear space.

\subsubsection{Approximation properties of the solution manifold}

We briefly recall the notion of Kolmogorov $N$-width and its relevance for reduced-order modeling of parametrized PDEs.

\begin{definition}[Solution manifold and Kolmogorov $N$-width]
Let $V$ be a Hilbert space (e.g., $V = H^s_0(\Omega)$), and let
\[
   \mathcal M := \{ u(\mu) : \mu \in \mathcal D\} \subset V
\]
denote the solution manifold associated with a parametrized PDE, where
$\mathcal D \subset \mathbb R^P$ is a compact parameter set
(for instance, $\mu$ encodes forcing terms, coefficients, or boundary data).
The Kolmogorov $N$-width of $\mathcal M$ in $V$ is defined by
\[
  d_N(\mathcal M;V)
  :=
  \inf_{\substack{V_N \subset V\\ \dim V_N = N}}
  \;
  \sup_{u \in \mathcal M} \;
    \inf_{v \in V_N} \|u - v\|_V.
\]
It measures the best possible worst-case error attainable by any $N$-dimensional linear approximation space.
\end{definition}

We are interested in situations where $d_N(\mathcal M;V)$ decays faster than any algebraic rate in $N$.
For coercive linear elliptic problems with analytic parameter dependence this is known to hold.

\begin{assumption}[Coercive analytic parametric dependence]\label{ass:analytic}
Let $a(\cdot,\cdot;\mu)$ denote a bilinear form on $V \times V$ and $f(\cdot;\mu)\in V'$ the right-hand side of a variational formulation
\[
  a(u(\mu),v;\mu) = f(v;\mu) \qquad \forall v \in V,\ \mu\in\mathcal D.
\]
Assume that:
\begin{enumerate}
  \item (Uniform coercivity and boundedness)
        There exist $\alpha,\gamma>0$ such that for all $\mu\in\mathcal D$,
        \[
           \alpha \|v\|_V^2 \le a(v,v;\mu), \qquad
           |a(w,v;\mu)| \le \gamma \|w\|_V \|v\|_V,
           \quad \forall v,w\in V.
        \]
  \item (Affine and analytic parameter dependence)
        The bilinear form and the data admit an affine expansion
        \[
          a(w,v;\mu) = \sum_{q=1}^Q \Theta_q(\mu)\,a_q(w,v), \qquad
          f(v;\mu)   = \sum_{q=1}^{Q_f} \Theta^f_q(\mu)\,f_q(v),
        \]
        where the coefficient functions $\Theta_q,\Theta^f_q:\mathcal D\to\mathbb R$ extend analytically
        to a complex neighbourhood of $\mathcal D$.
\end{enumerate}
\end{assumption}

Under Assumption~\ref{ass:analytic}, $\mathcal M$ is a compact, analytically parametrized manifold in $V$.
This structure leads to very fast decay of the Kolmogorov $N$-width.

\begin{theorem}[Sub-exponential decay of $d_N(\mathcal M;V)$]\label{thm:nwidth}
Suppose Assumption~\ref{ass:analytic} holds. Then there exist constants
$C,c>0$ and an exponent $\alpha\in(0,1]$, depending only on the parameter dimension
and the analytic regularity, such that
\begin{equation}\label{eq:nwidth-subexp}
  d_N(\mathcal M;V) \le C \exp\bigl(-c N^{\alpha}\bigr), \qquad N\ge 1.
\end{equation}
In particular, for single-parameter problems ($P=1$) one typically has
$\alpha = 1$ and thus a purely exponential decay $d_N(\mathcal M;V)\le C e^{-cN}$,
while for multi-parameter problems one obtains a root-exponential (or, more generally,
sub-exponential) decay of the form $d_N(\mathcal M;V)\le C e^{-c N^{1/P}}$ (up to
logarithmic factors).
\end{theorem}

\begin{remark}[Implications for the empirical-kernel ROM]
The empirical kernel space $U_N := \mathcal H_N = \mathrm{span}\{u_1,\dots,u_N\}$
is an $N$-dimensional linear approximation space for the solution manifold $\mathcal M$.
If the snapshot solutions $\{u_i\}$ are chosen so as to approximate $\mathcal M$
(e.g., by a greedy reduced-basis strategy), the resulting spaces $U_N$ are known
to realize the same asymptotic rate of decay as the Kolmogorov $N$-width
(see, e.g., \cite{binev2011greedy,devore2012greedy,haasdonk2013podgreedy}).
Concretely, if $d_N(\mathcal M;V)\lesssim \exp(-cN^{\alpha})$,
then the worst-case approximation error
\[
  \sup_{\mu\in\mathcal D} \inf_{v\in U_N} \|u(\mu)-v\|_V
\]
inherits the same sub-exponential decay in $N$, up to different constants.
Since our Gaussian–process based solver performs an optimal recovery
in $U_N$, the numerical
errors we observe as $N$ increases are consistent with the sub-exponential
(or even exponential) decay predicted by Theorem~\ref{thm:nwidth}.
\end{remark}

As mentioned in the previous remark, we obtain an (sub)-exponential decay of the error in terms of the number of snapshots if the snapshots are chosen by a greedy reduced-basis strategy for the elliptic problem induced by the bilinear form $a$. However, it remains an open question whether this theoretical bound can be proven for general nonlinear problems as we have observed in our numerical experiments. We also observe the (sub)-exponential decay when choosing the snapshots randomly.

\subsection{Error estimates in the sparsity parameter\texorpdfstring{$\rho$}{rho}}

Our numerical experiments also show that the error decays (sub)-exponentially in the sparsity parameter $\rho$. This has been proven for the Matérn kernel in \cite{yifan2025}. However, for other kernels, in particular the empirical kernel, this remains an open question and will be investigated in future works. 

\section{Conclusion}
We introduced \emph{KROM}, a kernelized reduced-order modeling framework for fast solving of nonlinear PDEs that combines two complementary ideas: (i) \emph{empirical kernels} constructed from snapshot libraries to encode problem-specific solution structure, and (ii) \emph{sparse Cholesky factorization} of the precision matrix to make Gaussian-process optimal recovery scalable. Unlike classical ROMs that require intrusive projection and hyper-reduction and the classical GP for PDE method that requires the right choice of the kernel and hyperparameter tuning, KROM remains non-intrusive and enforces the governing equations through operator constraints, while sparsification yields an implicit reduced model by selecting only a small, geometrically localized set of effective degrees of freedom. Because the sparse Cholesky algorithm is based on a rank-revealing max–min ordering of the collocation points, it inherently functions as a built-in reduced-order model.

Across a range of benchmark nonlinear PDEs (semilinear elliptic problems, discontinuous-coefficient Darcy flow, viscous Burgers, Allen--Cahn, and two-dimensional Navier--Stokes), we observed that snapshot-driven kernels can substantially outperform generic kernels such as Mat\'ern or Gaussian kernel in regimes with sharp gradients, discontinuities, or strongly anisotropic dynamics. This is consistent with the operator-theoretic interpretation of empirical kernels as finite-rank surrogates of regularized inverses and with the fact that homogeneous linear constraints (such as boundary conditions) can be satisfied by construction when they are shared by the snapshots. Our theory complements these observations by providing an explicit error decomposition that separates (a) discretization/collocation effects, (b) snapshot/modeling error due to the finite empirical kernel space, and (c) sparse-Cholesky approximation error controlled by the sparsity radius.

Several directions are natural next steps. On the modeling side, it is important to develop principled snapshot selection strategies (e.g., greedy or active learning) and to extend the kernel construction to efficiently accommodate parameterized geometries and inhomogeneous constraints via lifting or kernel augmentation. On the algorithmic side, adaptive choices of collocation sets and sparsity patterns, as well as multi-level or domain-decomposed variants of sparse Cholesky, could further improve robustness for multi-scale flows. Finally, incorporating noise models and posterior uncertainty into goal-oriented error indicators would enable reliable adaptive refinement and real-time digital-twin applications. 

\section{Acknowledgments}

The authors acknowledge support from the Air Force Office of Scientific Research under MURI award number FOA-AFRL-AFOSR-2023-0004 (Mathematics of Digital Twins), the Department of Energy under award number DE-SC0023163 (SEA-CROGS: Scalable, Efficient, and Accelerated Causal Reasoning Operators, Graphs and Spikes for Earth and Embedded Systems), the National Science Foundation under award number 2425909 (Discovering the Law of Stress Transfer and Earthquake Dynamics in a Fault Network using a Computational Graph Discovery Approach). HO acknowledges support from the DoD Vannevar Bush Faculty Fellowship Program under ONR award number N00014-18-1-2363. 

\bibliographystyle{siam} 
\bibliography{references.bib}

\appendix

\section{Reproducing Kernel Hilbert Space}
In this section, we want to state some facts about reproducing kernel Hilbert spaces (RKHS) and the representer formula in the general case and the special case where the RKHS space is generated by the empirical kernel constructed above.

In the following result, we characterize the RKHS space generated by the empirical kernel.  

\begin{theorem}
    Let $\Omega\subset \mathbb{R}^d$ be a bounded Lipschitz domain and let $\mathcal{F}$ be a suitable function space. Furthermore, let $\lbrace u_i\rbrace_{i=1}^N\subset \mathcal{F}\backslash \lbrace 0\rbrace$ be a set of linearly independent functions and $k:\Omega\times \Omega\rightarrow \mathbb{R}$ be given by $k(x,y)=\frac{1}{N} \sum_{i=1}^N u_i(x)u_i(y)$.
Then, the following statements hold:
\begin{itemize}
\item[$i)$] The function $k$ defines a kernel. 
\item[$ii)$] The reproducing kernel Hilbert space $\mathcal{H}$ generated by the kernel $k$ satisfies $\mathcal{H}=\mathrm{span}\lbrace u_1, u_2, \dots, u_N \rbrace$.
\item[$iii)$] For a functions $f,g\in \mathcal{H}$ with $f=\sum_{i=1}^N \alpha_i u_i$ and $g=\sum_{i=1}^N \beta_i u_i$ for values $\alpha_i, \beta_i\in \mathbb{R}$, the inner product on $\mathcal{H}$ is given by $\langle f,g\rangle_{\mathcal{H}}= N \sum_{i=1}^{N} \alpha_i \beta_i$.
\end{itemize}
\end{theorem}

\begin{proof}
    Part $i)$: Symmetry is straightforward. Positive definiteness comes from the fact that the empirical kernel matrix $K$ with entries $K_{ij} = k(x_i,x_j) = \sum^N_{k=1} u_k (x_i) u_k (x_j)$ can be decomposed as $K=UU^\top$, where $U_{ik} = \frac{1}{\sqrt{N}} u_k (x_i)$. Then for any vector $c \in \mathbb{R}^M$, $c^\top K c = ||U^\top c||^2 \geq 0$ with equality if and only if $c=0$. Hence $k$ is a reproducing kernel generated by the linear span of $\lbrace u_i\rbrace_{i=1}^N$ with a unique RKHS.
    
    Part $ii)$: To prove that $\mathcal{H} = \textrm{Span}\{u_1,\dots,u_N\}$, we must first show that any function of the form $\sum^{N}_{i=1} \alpha_i u_i$, where $\alpha_i \in \mathbb{R}$, belongs to $\mathcal{H}$; and that all functions $f \in \mathcal{H}$ can also be decomposed as $f = \sum^{N}_{i=1} \alpha_i u_i$. The forward direction follows directly from the definition of the span. The reverse direction is a corollary of the representer formula \eqref{eq.representer} for the empirical kernel: 

    By the representer theorem, the minimizer $u^\dagger$ of \eqref{eq.min_fin} in $\mathcal H_N$ admits an expansion of the form
\[
  u^\dagger(\cdot) = \sum_{j=1}^M \alpha_j\, k_N(x_j,\cdot),
\]
where $\{x_j\}_{j=1}^M$ denote the locations (or more generally the linear functionals) entering the constraints and $\alpha_j\in\mathbb R$. For each fixed $x_j$, the empirical kernel takes the form
\begin{align}\label{eq:kernel.form}
  k_N(x_j,x)
  = \frac{1}{N}\sum_{i=1}^N u_i(x_j)\,u_i(x),
  \qquad x\in \Omega,
\end{align} 
so $k_N(x_j,\cdot)$ is a linear combination of $\{u_i\}_{i=1}^N$ and hence belongs to
$\mathrm{span}\{u_1,\dots,u_N\}$.
Therefore,
\[
  u^\dagger(\cdot)
  = \sum_{j=1}^M \alpha_j k_N(x_j,\cdot)
  = \sum_{j=1}^M \alpha_j \Bigl( \frac{1}{N}\sum_{i=1}^N u_i(x_j)\,u_i(\cdot)\Bigr)
  = \sum_{i=1}^N c_i u_i(\cdot),
\]
where
\(
  c_i = \frac{1}{N}\sum_{j=1}^M \alpha_j u_i(x_j).
\)
Thus $u^\dagger \in \mathrm{span}\{u_1,\dots,u_N\}$, which proves the claim.

    Part $iii)$: Finally, we verify that $\mathcal{H}$ is indeed a Hilbert space, we use the reproducing property. Let $f = \sum^N_{i=1} \alpha_i u_i, g = \sum^N_{i=1} \beta_i u_i$ for $f,g \in \mathcal{H}$. Since $u_1,\dots,u_N$ are linearly independent, we can use the reproducing property to write $f (\cdot) = \langle w, \Phi (\cdot) \rangle$ and $g (\cdot) = \langle v, \Phi (\cdot) \rangle$, where $w = \sqrt{N} \alpha, v = \sqrt{N} \beta, \Phi = \frac{1}{\sqrt{N}}(u_1,\dots,u_N)$ , hence

    \begin{equation}
    \langle f,g, \rangle_\mathcal{H} = \langle \sqrt{N} \alpha, \sqrt{N} \beta \rangle = N \sum^N_{i=1} \alpha_i \beta_i
    \end{equation}

    which is an inner product on $\mathcal{H}$.
\end{proof}

In the following, we present a very relevant and prominent example.

\begin{example}
    Let $\mathcal{L}: \mathrm{H}_0^s(\Omega)\rightarrow \mathrm{H}^{-s}(\Omega)$ be a linear bounded operator such that $\mathcal{L}$ as a complete and orthogonal set of eigenfunctions denoted by $\mathcal{B}:= \lbrace \phi_1, \phi_2, \dots, \rbrace$ and denote with $\lbrace \lambda_1, \lambda_2, \dots, \rbrace$ the corresponding eigenvalues. Then, defining for a fixed $N \in \mathbb{N}$ the empirical kernel by the finite set of eigenfunctions $\mathcal{B}_N:\lbrace \phi_1, \phi_2, \dots, \phi_N\rbrace$. Then, it is well known that the scaled empirical kernel 
    $k_n(x,y)=\sum_{i=1}^N \frac{1}{\lambda_i^s} \phi_i(x)\phi_i(y)$ converges to the Green's function associated with the operator $\mathcal{L}^{s/2}$. In the special case where $\mathcal{L}=(-\Delta)^{s}$ denotes the fractional Laplacian, the RKHS $\mathcal{H}=H^{2s}_0 (\Omega)$ coincides with the fractional Sobolev space. 
\end{example}

Next, we state the so-called representer theorem which shows that the infinite dimensional optimal recovery \eqref{eq:min_norm_main} problem reduces to the infinite dimensional optimization problem \eqref{eq.min_fin}.

\begin{theorem}[Representer Theorem] \label{th:representer.thm}
Let $\mathcal{H}$ be a Hilbert space with inner product $\langle \cdot, \cdot \rangle_{\mathcal{H}}$ and corresponding norm $\| \cdot \|_{\mathcal{H}}$. Let $\{x_i\}_{i=1}^n \subset \mathcal{X}$ be a set of given data points, and consider the optimization problem:
\begin{align}
    \min_{f \in \mathcal{H}} \mathcal{L}(\{f(x_i)\}_{i=1}^n) + \lambda \|f\|_{\mathcal{H}}^2,
\end{align}
where:
\begin{itemize}
    \item $\mathcal{L}: \mathbb{R}^N \to \mathbb{R}$ is a loss function depending on the values $\{f(x_i)\}_{i=1}^n$,
    \item $\lambda > 0$ is a regularization parameter
\end{itemize}

Assume that
\begin{enumerate}
    \item $\mathcal{H}$ is a reproducing kernel Hilbert space (RKHS) with kernel $k: \mathcal{X} \times \mathcal{X} \to \mathbb{R}$,
    \item The loss function $\mathcal{L}$ depends only on the evaluations $\{f(x_i)\}_{i=1}^N$, and is differentiable and convex with respect to these evaluations.
\end{enumerate}

Then the minimizer $f^* \in \mathcal{H}$ of the optimization problem can be expressed as:
\begin{align} \label{eq.representer}
    f^*(x) = \sum_{i=1}^N \alpha_i k(x_i, x),
\end{align}
where $\{\alpha_i\}_{i=1}^N \subset \mathbb{R}$ are coefficients that depend on the specific loss function $\mathcal{L}$ and the regularization parameter $\lambda$.
\end{theorem}
Here we can see that $\mathcal{X}$ is an arbitrary set. Therefore, instead of choosing $\{x_i\}_{i=1}^N \subset \mathcal{X}$ to be points in the domain $\Omega$, we can choose $\{\phi_i\}_{i=1}^N \subset X^*$ to be linear bounded functionals that are in the dual space of a given Hilbert space $\mathcal{H}$.

We recall that the choice of the kernel in \eqref{eq:kernel.form} corresponds to a local measurements of the functions which leads to a pointwise evaluation of the solutions after applying the representer theorem. However, in most situations, the solutions are not classical which does not permit a pointwise evaluation of the solution. Nevertheless, our GP framework allows for a more general choice of the kernel corresponding to more global measurements of comparing two functions. Using these kernels, we can compare functions that are not continuous and are more suitable for weak solutions. More precisely, we will make use of the following representer theorem:

\section{Sparse Cholesky Factorization}

    The primary challenge of using a naive Gauss-Newton optimization to solve the finite-dimensional problem \eqref{eq.min_fin} is that many algorithms for solving linear equations and inverting the matrix $K(\phi,\phi)$ are of order $\mathcal{O}(N^3)$ complexity in time/space, which makes the PDE solver prohibitively expensive for large numbers of collocation points. To alleviate this issue, we use the sparse Cholesky algorithm introduced in \cite{schaefer2021} and \cite{yifan2025} for approximating the Cholesky factor of $K(\phi,\phi) ^{-1}$ with a space complexity of $\mathcal{O}(N \log^d (N/\epsilon))$ and a time complexity of $\mathcal{O}(N \log^{2d} (N/\epsilon))$. The algorithm returns a permutation matrix $P_{\textrm{Perm}}$ and a sparse upper triangular matrix $U$ with $\mathcal{O}(N \log^d (N\epsilon))$ nonzero entries. The matrices $U$ and $P_{\textrm{Perm}}$ satisfy the error bound, where $||\cdot||_{\textrm{Fro}}$ is the Frobenius norm:

    \begin{equation}
        ||K(\phi,\phi)^{-1}-P^\top_{\textrm{Perm}} UU^\top P_{\textrm{Perm}}||_{\textrm{Fro}} \leq \epsilon.
    \end{equation}
    
    The sparse Cholesky algorithm is based on a maximin (or a coarse-to-fine) reordering of the collocation points $\{\mathbf{x}_i\}_{i \in I}, I = \{1 \dots, M \}$. This works by selecting the first point at random and successively selecting the point $\mathbf{x}_j$ that is furthest away from the already picked points. 
    We list the indices of these ordered points as 
    
    \begin{equation}
    i_{q+1} = \arg \max_{i  \in I \texttt{\textbackslash} \{i_1,\dots,i_q \}} \textrm{dist} (\mathbf{x}_i, \{\mathbf{x}_{i_1},\dots \mathbf{x}_{i_q}\})
    \end{equation}

    After ordering our Dirac measurements under this maximin ordering, we then order the derivative measurements arbitrarily. We define the permutation map $P:I_N \rightarrow I, P(q) = i_q$, where $I_N = \{1 \dots, M_{\partial \Omega}+DM_{\Omega} \}$, which outputs the maximin index of the $q$-th ordered point for the Dirac measurements and our chosen ordering for the derivative measurements. This maximin ordering results in a geometric covering of the $d$-dimensional domain: after $j$ steps in the maximin ordering, the first $j$ reordered collocation points are well-separated and cover the domain with cells of diameter $\epsilon_i \approx c j^{-\frac{1}{d}}$. For any new point $\mathbf{x}_j$, only $\mathcal{O}(1)$ earlier points lie within a ball of radius $C \epsilon_i$. 

    The off-diagonal entries of the Cholesky factor of $K^{-1}$ exhibit exponential decay under this maximin reordering; this leads to the idea of 'zeroing out' these negligible entries to create sparse approximate Cholesky factors. To see why, recall that we first approximate our unknown function $v$ by a Gaussian Process $\xi \sim \mathcal{N}(0,K)$ and define the corresponding Gaussian random variables $Y_i:= [\xi, \phi_i] \sim N(0,K(\phi_i,\phi_i))$. Then we derive the following relation

    \begin{equation}
        \frac{U_{ij}}{U_{jj}} = (-1)^{i \neq j}  \frac{\textrm{Cov}[Y_i,Y_j|Y_{1:j-1 \texttt{\textbackslash}   \{i\}}]}{\textrm{Var}[Y_i|Y_{1:j-1 \texttt{\textbackslash} \{i\}}]},
      \label{eq:screening}
    \end{equation}

    where $U= L^T$. Formula \eqref{eq:screening} links the values of U to the conditional variance of a Gaussian process, and is related to the screening effects in spatial statistics, in which conditioning on GP values at a dense enough set of points between two regions results in points in one region becoming almost independent of points in the other.
    
    We introduce the lengthscale $l_i$, defined as follows:
    
    \begin{equation}
    l_i = \begin{cases}
    \infty, \; i=1 \\
    \min_{j<i} ||\mathbf{x}_{P(i)}-\mathbf{x}_{P(j)}||, \; 2 \leq i \leq  M \\
    l_M, \; i>M
    \end{cases}
    \end{equation}

    For the Dirac measurements, the value of $l_i$ is the distance from each ordered point to the previously ordered point; a smaller $l_i$ indicates a faster decay of element $U_{ij}, i \leq j$ from the diagonal $U_{jj}$.
    
    To control how sparse we want our matrices $U$ to be, we introduce a small hyperparameter $\rho>0$ such that the number of nonzero entries in each row of $U$ is the size of the set $S_{P,l,\rho} = \{(i,j) \subset I_N \times I_N: i\leq j, ||\mathbf{x}_{P(i)},\mathbf{x}_{P(j)}||\leq \rho l_j\} \subset I_N \times I_N$. We use the denotation $s_j = \{i: (i,j) \in S_{P,l,\rho}\}$ to denote the row elements for a given column $j$ of $U$ that are nonzero and $|s_j|$ to represent the number of elements in $s_j$.

    Finally, defining $\Theta = K(\phi,\phi)$, we seek a matrix $U \in S^{\textrm{mtx}_{P,l,\rho}}$ such that the following Kullback-Leibler (KL) divergence is minimized:

    \begin{equation}
        U = \arg \min_{\hat{U} \in S^{\textrm{mtx}_{P,l,\rho}}} \textrm{KL} (\mathcal{N} (0,\Theta)|| \mathcal{N} (0,(\hat{U}\hat{U}^\top)^{-1})).
        \label{eq.KL_min}
    \end{equation}

    The solution to \eqref{eq.KL_min} has the explicit formula where $U$ is defined columnwise as

    \begin{equation}
        U_{s_j,j} = \frac{\Theta^{-1}_{s_j,s_j} \mathbf{e}_{|s_j|}}{\sqrt{\mathbf{e}^\top_{|s_j|} \Theta^{-1}_{s_j,s_j} \mathbf{e}_{|s_j|}}},
    \end{equation}

    where we define $\mathbf{e}_{|s_j|}$ as the canonical basis vector in $\mathbb{R}^{|s_j|}$ with the last entry being 1, and $\Theta^{-1}_{s_j,s_j} = (\Theta_{s_j,s_j})^{-1}$, where $\Theta_{s_j,s_j}$ is the submatrix of $\Theta$ formed by extracting the entries where both the row and column number belong to the set $s_j$.

\section{Proofs} 
\label{sec:appendix_proofs}

\begin{proof}[Proof of Theorem \ref{thm:empirical_kernel_rate}]
We adapt the argument of Theorem~1.2 to the changing RKHS $U_N$, and we isolate the additional approximation error arising from restricting to the snapshot space.

\textit{1. Decomposition into discretization and snapshot-space bias.}
Let $u^\star$ be the true solution of \eqref{eq:ibvp_main}. Let $u_N^\star\in U_N$ be any function in $U_N$ satisfying the \emph{continuous} constraints of \eqref{eq:ibvp_main} (if $u^\star\in U_N$, choose $u_N^\star=u^\star$). Then by the triangle inequality,
\begin{equation}\label{eq:tri_split}
\|u_{N,h}^\dagger-u^\star\|_{H^\ell(\Omega)}
\le
\|u_{N,h}^\dagger-u_N^\star\|_{H^\ell(\Omega)}
+
\|u_N^\star-u^\star\|_{H^\ell(\Omega)}.
\end{equation}
The second term is bounded by $\inf_{v\in U_N}\|v-u^\star\|_{H^\ell(\Omega)}$ by the definition of best approximation, so it remains to bound the first term.

\textit{Second, apply Theorem~1.2 with $u_N^\star$ as the ``target'' in $U_N$.}
Since $u_N^\star\in U_N$ satisfies the continuous PDE constraints, and $u_{N,h}^\dagger$ satisfies the discrete (collocated) constraints by construction, the residual
\[
\bar f := \mathcal{P}(u_{N,h}^\dagger)-\mathcal{P}(u_N^\star)
\]
vanishes at the interior collocation points (and similarly for boundary/initial residuals if these are enforced by collocation). Under the stability assumptions of Theorem~1.2 and the sampling inequality (or Bramble--Hilbert type estimate) used there, one obtains for sufficiently small $h_\Omega$ an estimate of the form
\begin{equation}\label{eq:residual_to_error}
\|u_{N,h}^\dagger-u_N^\star\|_{H^\ell(\Omega)}
\le
C\, h_\Omega^\gamma\, \|u_N^\star\|_{U_N},
\end{equation}
with constants $C,h_0$ independent of $h_\Omega$ and the minimizer. 
The only additional requirement relative to Theorem~1.2 is that the embedding $U_N\hookrightarrow H^s(\Omega)$ needed to control higher derivatives of the residual holds; this is automatic for fixed $N$ (finite-dimensional norm equivalence) and is uniform in $N$ under Assumption~\ref{ass:uniform_embedding}.

Combining \eqref{eq:tri_split} and \eqref{eq:residual_to_error} gives \eqref{eq:empirical_two_term}.

\textit{Step 3: Uniform boundedness in $N$.}
Assume \ref{ass:bounded_comparator}. Since $u_{N,h}^\dagger$ minimizes the $U_N$-norm over the discrete constraint set, it satisfies the standard optimal-recovery inequality
\[
\|u_{N,h}^\dagger\|_{U_N} \le \|\tilde v_N\|_{U_N}\le \sup_N\|\tilde v_N\|_{U_N}<\infty.
\]
If additionally $u_N^\star$ is chosen as the $U_N$-minimum-norm element among continuous-feasible functions in $U_N$, then similarly $\|u_N^\star\|_{U_N}\le \|\tilde v_N\|_{U_N}$, hence $\sup_N\|u_N^\star\|_{U_N}<\infty$.

\textit{Step 4: Convergence as $N\to\infty$.}
Let $h_\Omega=h_\Omega(N)\downarrow 0$ and combine \eqref{eq:empirical_two_term} with Assumption~\ref{ass:density} and the uniform boundedness of $\|u_N^\star\|_{U_N}$ to obtain
\[
\|u_{N,h_\Omega(N)}^\dagger-u^\star\|_{H^\ell(\Omega)}
\le
C\, h_\Omega(N)^\gamma\,\|u_N^\star\|_{U_N}
+
\inf_{v\in U_N}\|v-u^\star\|_{H^\ell(\Omega)}\xrightarrow[N\to\infty]{}0,
\]
provided $h_\Omega(N)^\gamma\|u_N^\star\|_{U_N}\to 0$. This proves the stated convergence. 
Finally, if $u^\star\in U_N$ for all sufficiently large $N$, then the bias term in \eqref{eq:empirical_two_term} vanishes, yielding \eqref{eq:empirical_single_term}.
\qed
\end{proof}

\end{document}